\documentclass{article}

\usepackage{stylefiles/packages}
\usepackage{stylefiles/format}
\usepackage{stylefiles/abbrev}

\providecommand{\keywords}[1]{\textbf{\textit{Keywords :}} #1} 
\providecommand{\MSC}[1]{\textbf{\textit{MSC2020 :}} #1} 

\title{Phase transitions for support recovery under local differential privacy}

\author[,1]{Cristina Butucea\thanks{Financial support from the French National Research Agency (ANR) under the grant Labex Ecodec (ANR-11-LABEX-0047)} }
\author[,2]{Amandine Dubois \thanks{Corresponding author. Financial support from GENES and from the French ANR grant ANR-18-EURE-0004} }
\author[2]{Adrien Saumard \thanks{Financial support from the French ANR grant ANR-18-EURE-0004}}
\affil[1]{CREST, ENSAE, Institut Polytechnique de Paris, 5 avenue Henry Le Chatelier, F-91120 Palaiseau.
cristina.butucea@ensae.fr}
\affil[2]{CREST, ENSAI, Campus de Ker-Lann - Rue Blaise Pascal - BP 37203 - 35172 BRUZ cedex. amandine.dubois@ensai.fr, adrien.saumard@ensai.fr}
\date{}
\begin{document}
\maketitle
\begin{abstract}
We address the problem of variable selection in a high-dimensional but sparse mean model, under the additional constraint that only privatised data  are  available for inference. The original data are vectors with independent entries having a symmetric, strongly log-concave distribution on $\Rb$.
For this purpose, we adopt a recent generalisation of classical minimax theory to the framework of local $\alpha-$differential privacy. We provide lower and upper bounds on the rate of convergence for the expected Hamming loss over classes of at most $s$-sparse vectors whose non-zero coordinates are separated from $0$ by a constant $a>0$.
As corollaries, we derive necessary and sufficient conditions (up to log factors) for exact recovery and
for almost full recovery.
When we restrict our attention to non-interactive mechanisms that act independently on each coordinate our lower bound shows that, contrary to the non-private setting, both exact and almost full recovery are impossible whatever the value of $a$ in the high-dimensional regime such that $n \alpha^2/ d^2\lesssim 1$.
However, in the regime $n\alpha^2/d^2\gg \log(d)$ we can exhibit a critical value $a^*$ (up to a logarithmic factor) such that exact and almost full recovery are possible for all $a\gg a^*$ and impossible for $a\leq a^*$.
We show that these results can be improved when allowing for all non-interactive (that act globally on all coordinates) locally $\alpha-$differentially private mechanisms in the sense that phase transitions occur at lower levels.
\end{abstract}


\MSC{62G05, 62G20}
\section{Introduction}

We consider the problem of distributed support recovery of the sparse mean of $n$ independent, identically distributed (i.i.d.) random vectors.
Precisely, for $i=1,\ldots,n$, the $i$th data holder observes a random vector $X^i=(X^{i}_j)_{j=1,\ldots,d}\in\Rb^d$ issued from a rescaled and shifted vector $\xi^i$: $X^i = \theta+ \sigma \xi^i$. The noise is supposed to have independent coordinates $\xi^i_j, \, j=1,...,d$ identically distributed with a symmetric and strongly log-concave distribution of variance 1 (see Section \ref{subsec:SLC} below for definition and details). Note that the standard Gaussian distribution belongs to our model, but the symmetric and strongly log-concave probability density functions form a large non-parametric class of functions.

The mean vector $\theta$ is assumed to be $(s,a)$-sparse in the sense that $\theta$ belongs to one of the following sets:
\begin{multline*}
\Theta_d^+(s,a)=\{\theta \in \Rb^d: \text{there exists a set } S\subseteq\{1,\ldots,d\} \text{ with at most } s \text{ elements }\\ \text{ such that }
\theta_j\geq a \text{ for all } j\in S, \text{ and } \theta_j=0 \text{ for all } j\notin S \},
\end{multline*}
or
\begin{multline*}
\Theta_d(s,a)=\{\theta \in \Rb^d: \text{there exists a set } S\subseteq\{1,\ldots,d\} \text{ with at most } s \text{ elements }\\ \text{ such that }
\vert\theta_j\vert \geq a \text{ for all } j\in S, \text{ and } \theta_j=0 \text{ for all } j\notin S \}.
\end{multline*}

\subsection{Differential Privacy}\label{Section Problem statement}

Nowadays, a large amount of data, such as internet browsing history, social media activity, location information from smartphones, or medical records, are collected and stored.
On the one hand, the analysis of these data can benefit to individuals, companies, or communities such as the scientific one.
For instance, companies can use data to improve their products and services, or health data can be used for medical research.
On the other hand, people are more and more concerned with the protection of their privacy and may be reluctant to share their sensitive data.
In this context, it seems essential to be able to understand the tradeoffs between the statistical utility of the collected data and the privacy of individuals from whom these data are obtained.
This requires a formal definition of privacy and differential privacy has been adopted by researchers in the computer science, machine learning, and statistics communities as a natural one.

Two kinds of differential privacy are discussed in the literature: central differential privacy which has been introduced by Dwork et al. in \cite{DworkMcSherryNissimSmith2006}, and local differential privacy.
We will focus in this paper on the second setting but we briefly discuss the difference between central and local privacy.
In both settings, $n$ individuals want their privacy to be preserved while their data, which will be denoted $X_1,\ldots,X_n$, are used for statistical analyses.
In the central setting, the $n$ data-holders share confidence in a common curator who has access to the original data $X_1,\ldots,X_n$ and use them to generate a private release $Z$.
In a nutshell, central differential privacy ensures that the probability of observing an output does not change much when a single data point of the original database is modified.
We refer to \cite{WassermanZhou2010differentialPrivacy} for the formal definition of differential privacy in the central setting.
In the local setting, data is privatized before it is shared with a data collector : for all $i\in\llbr 1,n\rrbr$, $X_i$ is transformed into a private data $Z_i$ directly on the $i$th individual's machine and the data collector or the statistician only have access to the private sample $Z_1,\ldots, Z_n$.
However, some interaction between the different data-holders is allowed.
Formally, the privatized data $Z_1,\ldots,Z_n$ are obtained by successively applying suitable Markov kernels :
given $X_i=x_i$ and $Z_1=z_1,\ldots,Z_{i-1}=z_{i-1}$, the $i$-th dataholder draws
\begin{equation*}
    Z_i \sim Q_i(\cdot \mid X_i=x_i, Z_1=z_1,\ldots,Z_{i-1}=z_{i-1})
\end{equation*}
for some Markov kernel $Q_i: \Zs \times \Xc \times \Zc^{i-1} \to [0,1]$ where the measure spaces of the non-private and private data are denoted with $(\Xc,\Xs)$ and $(\Zc,\Zs)$, respectively.
Such randomizations are known as sequentially interactive.
We say that the sequence of Markov kernels $(Q_i)_{i=1,\ldots,n}$ provides $\alpha$-local differential privacy or that $Z_1,\ldots,Z_n$ are $\alpha$-local differentially private views of $X_1,\ldots,X_n$ if
\begin{equation}\label{alphaLDPconstraint}
    \sup_{A \in \Zs} \frac{Q_i(A \mid X_i=x,Z_1=z_1,\ldots,Z_{i-1}=z_{i-1})}{Q_i(A \mid X_i=x^\prime,Z_1=z_1,\ldots,Z_{i-1}=z_{i-1})} \leq \exp(\alpha) \quad \forall i\in\llbr 1,n\rrbr, \, \forall x, x^\prime \in \Xc.
\end{equation}
In this paper, we will focus on the special case of non-interactive local differential privacy where $Z_i$ depends only on $X_i$ but not on $Z_k$ for $k<i$.
In this scenario, we have
$$
Z_i \sim Q_i(\cdot \mid X_i = x_i),
$$
and condition \eqref{alphaLDPconstraint} becomes
$$
\sup_{A \in \Zs} \frac{Q_i(A \mid X_i=x)}{Q_i(A \mid X_i=x^\prime)} \leq \exp(\alpha) \quad \forall i\in\llbr 1,n\rrbr, \, \forall x, x^\prime \in \Xc.
$$

The aim is that every data holder releases a private view $Z^{i}$ of $X^{i}$ such that the notion of local differential privacy is satisfied and that the support of $\theta$ can be estimated from the data $Z^{1},\ldots,Z^{n}$ in an optimal way.

{\bf Notation.} For two sequences $\{a_d\}_d$ and $\{b_d\}_d$ of non-negative real numbers, we write $a_d\lesssim b_d$ if there exists some constant $C>0$ such that $a_d\leq Cb_d$.
If $b_d >0$ we write $a_d\sim b_d$ if $a_d/b_d\rightarrow 1$ as $d\rightarrow \infty$ and we write $a_d \gg b_d$ if $a_d/b_d \rightarrow \infty$ as $d\rightarrow \infty$.
We recall that a centred Laplace distribution with parameter $\lambda>0$ has the probability density function defined by $f_\lambda(x)=\frac{1}{2\lambda}\exp\left(-\frac{\vert x\vert}{\lambda}\right)$ on $\Rb$.

\subsection{Motivation}

The problem of high-dimensional sparse vectors estimation has recently been studied in the framework of local differential privacy in \cite{DuchiJordanWainwright2018MinimaxOptimalProcedure}.
For the $1$-sparse mean estimation problem, the authors considered the set of distributions $P$ supported on $\Bb_\infty(r)$, i.e. the ball of  radius $r$ in $\Rb^d$ with respect to the sup norm $\|\cdot\|_\infty$, and having $\Vert \Eb_P[X]\Vert_0\leq 1$. They proved that the private minimax mean squared error for non-interactive $\alpha$-locally differentially private mechanisms is bounded from below by
$$
\min\left\{r^2,\frac{r^2d\log(2d)}{n(e^\alpha-1)^2}  \right\},
$$
proving that high-dimensional $1$-sparse mean estimation is impossible in this setting when both $r^2 \gtrsim 1$ and $r^2 d\log(2d) \gtrsim n(e^\alpha - 1)^2$. This result can be related to selecting the support of a 1-sparse mean vector of such a distribution $P$, under the same constraints. We generalize these results to symmetric and strongly log-concave distributions on the whole $\mathbb{R}^d$ and to arbitrary sparsity.



Obvious applications of variable selection are the estimation of the set that supports the non-null coefficients in the mean vector, or the estimation of its size.
We propose to use our procedure to build a private mean estimator of $s-$sparse vectors in two steps: use one part of the sample to recover the support and the other part to estimate the mean values of the selected variables, that is a vector of reduced size. Moreover these results are a benchmark for working on more realistic models such as  high-dimensional linear regression and clustering of high-dimensional vectors, see \cite{Ndaoud_Tsybakov_2020_optimal_variable_selection} and \cite{Ndaoud_2018_Gaussian_clustering}.

\subsection{Strongly log-concave distributions}\label{subsec:SLC}
Log-concave measures play a significant role in many areas of pure and applied mathematics, such as convex geometry \cite{Guedon:12}, functional inequalities \cite{Bobkov00}, optimal transport theory \cite{Caff92,Caff00}, random matrix theory \cite{Adam10}, Monte-Carlo sampling \cite{Dalal17,Durmus17}, Bayesian inference \cite{paninski:04} or non-parametric estimation \cite{Cule_Samworth_10,DossWellner:2013,Han21}. The log-concavity assumption arises also naturally in various modelisation contexts, such as survival and reliability analysis \cite{Jones15} or econometrics \cite{Bagnoli:95}, since it possesses many interesting properties subject to interpretation, such as monotone likelihood ratio or non-decreasing hazard rate function for instance. For further applications and references, see \cite{Bagnoli:95,SauWell14}.

Let us now state the definitions related to log-concavity that will be in force in this article. A probability distribution $P$ on $\mathbb{R}$ is log-concave if it admits a density $p$ with respect to the Lebesgue measure, that writes $p=\exp(-\phi)$, with $\phi$ a convex function on $\mathbb{R}$. The function $\phi$ is called the potential of the density $p$ and of the probability measure $P$.

Furthermore, a function $\phi:\mathbb{R} \to \mathbb{R}$ is $c-$strongly convex for some constant $c>0$ if, for all $(x,y)\in\mathbb{R}^{2}$ and $t\in(0,1)$,
\begin{equation}
\phi(t x+(1-t)y)- \left[t\phi(x) + (1-t)\phi(y) \right] \leq-\frac{c}{2} t(1-t)(x-y)^{2}.\label{eq:SC}
\end{equation}
 Note that the parameter $c$ in \eqref{eq:SC} gives a positive lower bound on the \textit{curvature} of the convex function $\phi$. In the case where the function $\phi$ is two times differentiable, condition \eqref{eq:SC} indeed corresponds to a lower bound on the second derivative: $\inf_{x\in \mathbb{R}}\left\{ \phi^{\prime\prime}(x)\right\}\geq c>0$.

 A probability measure $P$ is said to be $c-$strongly log-concave if it admits a density function $p:\mathbb{R}\to (0, + \infty)$ which is $c-$strongly log-concave with potential $\phi$, in the sense that $p=\exp(-\phi)$ and $\phi$ is a $c-$strongly convex potential. This is equivalent to assuming that $p(x)=\exp(-\phi_0(x))\exp(-cx^2/2)$, for all $x\in \mathbb{R}$, with $\phi_0$ being a finite convex function.

We consider the problem of support recovery of the sparse mean $\theta$ of a random vector $X=\theta + \sigma\xi$ of distribution $P_{\theta}$, where $\xi$ has i.i.d. coordinates $\xi_j$, $j=1,...,d$, distributed according to a $c-$strongly log-concave distribution $P^{\xi_1}$ for some constant $c>0$, with unit variance and that is symmetric around zero.  As $\xi_1$ is assumed to be symmetric, this amounts to require that the $c-$strongly convex potential $\phi$ of $p$ is even, or again that $p(x)=\exp(-\phi_0(x))\exp(-cx^2/2)$, for all $x\in \mathbb{R}$, where $\phi_0$ is a finite even convex function.

When dealing with some minimax lower bounds in the sequel, we will need to assume that the normalized noise distribution $p$ is not too peaked around its mean, in the sense that its curvature is bounded from above. More precisely, we will assume in this case that $p=\exp(-\phi)$, where $\phi$ is a finite convex potential satisfying, for a constant $c_+>0$, for all $(x,y)\in\mathbb{R}^{2}$ and $t\in(0,1)$,
\begin{equation}
\phi(t x+(1-t)y) - \left[ t \phi(x) + (1-t)\phi(y) \right] \geq-\frac{c_{+}}{2} t (1-t)(x-y)^{2}.
\label{eq:upper_bound_hessian_intro}
\end{equation}
When the potential $\phi$ is two times differentiable, condition \eqref{eq:upper_bound_hessian_intro} can be equivalently formulated as an upper bound on the second derivative of $\phi$: $\sup_{x\in \mathbb{R}}\left\{ \phi^{\prime\prime}\right\}\leq c_+$.

Such framework provides a non-parametric generalization of the Gaussian assumption, where $\phi_0$ would be assumed to be a constant function and the unit variance of $\xi$ would correspond to the value $c=c_+=1$. Note that when $\xi$ is only assumed to be centered and strongly log-concave, with unit variance and scaling parameter $c$, we have in general $c\leq 1$ and the equality case $c=1$ characterizes the normal distribution $\mathcal{N}(0,1)$, see \cite{Hillion_Johnsone_Saumard_18}. Finally, let us denote $\Phi$ the cumulative distribution function of the normal distribution.

\subsection{Minimax framework}
Let $X^{i}$, $i=1,\ldots,n$ be i.i.d random vectors of $\Rb^d$ {with distribution $P_{\theta}$}. We assume
that the vectors $X^{i}=(X^{i}_j)_{j=1,\ldots,d}$ for $i=1,\ldots,n$ are observed by $n$ distinct data holders who refuse to share their respective observations.
The statistician does not have access to these data but only to $\alpha$-locally differentially private views $Z^{1},\ldots Z^{n}$.
We assume that $\theta$ belongs to one of the sets $\Theta_d^+(s,a)$ or $\Theta_d(s,a)$ introduced in Section \ref{Section Problem statement} and we study the problem of selecting the relevant components of $\theta$, that is, of estimating the vector
$$
\eta=\eta(P_\theta)=\left(I(\theta_j\neq 0)\right)_{j=1,\ldots,d},
$$
where $I(\cdot)$ is the indicator function.
Our goal is to estimate the vector $\eta$ by a \textit{selector} $\hat{\eta}$, that is a measurable function $\hat{\eta}=\hat{\eta}(Z^{1},\ldots, Z^{n})$ taking values in $\{0,1\}^d$, where $Z^{1},\ldots, Z^{n}$ are $\alpha$-locally differentially private views of $X^{1},\ldots, X^{n}$.
We judge the quality of a selector $\hat{\eta}$ as an estimator of $\eta$ by the Hamming loss between $\hat{\eta}$ and $\eta$ which counts the number of positions at which $\hat{\eta}$ and $\eta$ differ :
$$
\vert \hat{\eta}-\eta\vert := \sum_{j=1}^d \vert \hat{\eta}_j-\eta_j\vert=\sum_{j=1}^d I(\hat{\eta}_j\neq \eta_j).
$$
For the support recovery problem, we consider only $\alpha$-locally differentially private mechanisms which transform  each $X^{i}\in\Rb^d$ into a  private release $Z^{i}$ taking also values in $\Rb^d$, that are known as non-interactive privacy mechanisms. However, we distinguish between privacy mechanisms that act on each coordinate of $X^i$ either separately, locally or globally.
More specifically, we will consider the two following scenarios:
\begin{itemize}
    \item {\bf Coordinate Local (CL) Privacy Mechanisms }: there is a sequence $Q=(Q^i)_{i=1,\ldots,n}$ of Markov kernels providing $\alpha$-local differential privacy such that $Z^i \sim Q^i(\cdot \mid X^i = x^i)$ for all $i\in\llbr 1,n\rrbr$, and $Q^i$ is obtained as product of coordinate-wise kernels as follows:
    $$
    \text{for all } i\in\llbr 1,n\rrbr \text{ and all } j\in\llbr 1,d\rrbr, Z^{i}_j\sim Q^{i}_j(\cdot \mid X^{i}_j=x)
    $$ for some $(\alpha/d)$-differentially private mechanism $Q^{i}_j$.
    We denote by ${\Qc}^{CL}_{\alpha}$ the set of all privacy mechanisms $Q=(Q^1,\ldots,Q^n)$ satisfying these assumptions.
    \item {\bf Coordinate Global (CG) Privacy Mechanisms }: there is a sequence $Q=(Q^i)_{i=1,\ldots,n}$ of Markov kernels providing $\alpha$-local differential privacy such that $Z^i \sim Q^i(\cdot \mid X^i = x^i)$ for all $i\in\llbr 1,n\rrbr$.
    We denote by $\Qc_{\alpha}$ the set of all privacy mechanisms $Q=(Q^1,\ldots,Q^n)$ satisfying this assumption.
\end{itemize}
In other words, in the Coordinate Local case, we consider only non-interactive $\alpha$-locally differentially private mechanisms that act coordinates by coordinates.
This scenario is easier to study than the second one for which any non-interactive $\alpha$-locally differentially private mechanism is allowed to be used.

For both scenarios, if $P_\theta$ denotes the distribution of  $X^{i}$ then we denote by $Q^{i}P_\theta$ the distribution of  $Z^{i}$.
Since the distribution of $(X^1,\ldots,X^n)$ is $P_\theta^{\otimes n}$, the distribution of $(Z^1,\ldots,Z^n)$ will be denoted by $Q(P_\theta^{\otimes n})$.
In the Coordinate Local case, we denote by $P_{\theta_j}$ the distribution of  $X^{i}_j$ and by $Q^{i}_jP_{\theta_j}$ the distribution of  $Z^{i}_j$.

We say that a selector $\hat{\eta}=(\hat{\eta}_1, \ldots, \hat{\eta}_d)$ is  \textit{separable} if for all $j=1,\ldots,d$ its $j$th component $\hat{\eta}_j$ depends only on $(Z^{i}_j)_{i=1,\ldots,n}$ .
We  denote by $\Tc$ the set of all separable selectors.
We are interested in the study of the following private minimax risks
\begin{equation}\label{MinimaxRisk scenario 1}
{\Rc}^{CL}_n(\alpha, \Theta)=\inf_{Q\in {\Qc}^{CL}_{\alpha}}\inf_{\hat{\eta}=\hat{\eta}(Z^{1},\ldots,Z^{n})\in\Tc}\sup_{\theta\in\Theta}\frac{1}{s}\Eb_{Q(P_\theta^{\otimes n})}\vert\hat{\eta}(Z^{1},\ldots,Z^{n})-\eta\vert,
\end{equation}
in the coordinate local case, and
\begin{equation}\label{MinimaxRisk scenario 2}
\Rc_n(\alpha, \Theta)=\inf_{Q\in \Qc_{\alpha}}\inf_{\hat{\eta}=\hat{\eta}(Z^{1},\ldots,Z^{n})\in\Tc}\sup_{\theta\in\Theta}\frac{1}{s}\Eb_{Q(P_\theta^{\otimes n})}\vert\hat{\eta}(Z^{1},\ldots,Z^{n})-\eta\vert,
\end{equation}
in the coordinate global case, for $\Theta=\Theta_d^+(s,a)$ and $\Theta=\Theta_d(s,a)$.

We are interested in the study of two asymptotic properties : \textit{almost full recovery} and \textit{exact recovery}, that we define here.
Let $(\Theta_d^+(s_d,a_d))_{d\geq 1}$ be a sequence of classes of sparse vectors.
We will say that \textit{almost full recovery is possible} for $(\Theta_d^+(s_d,a_d))_{d\geq 1}$ in the Coordinate Local case if there exists $Q\in {\Qc}^{CL}_\alpha$ and a selector $\hat{\eta}$ such that
$$
\lim_{d\rightarrow \infty} \sup_{\theta\in\Theta_d^+(s_d,a_d)}\frac{1}{s_d}\Eb_{Q(P_\theta^{\otimes n})}\vert\hat{\eta}-\eta\vert=0.
$$
We will say that \textit{almost full recovery is impossible} for $(\Theta_d^+(s_d,a_d))_{d\geq 1}$ in the Coordinate Local case if
$$
\liminf_{d\rightarrow +\infty} \inf_{Q\in {\Qc}^{CL}_{\alpha}}\inf_{\hat{\eta}=\hat{\eta}(Z^{1},\ldots,Z^{n})\in\Tc}\sup_{\theta\in\Theta_d^+(s,a)}\frac{1}{s_d}\Eb_{Q(P_\theta^{\otimes n})}\vert\hat{\eta}-\eta\vert >0.
$$
We will say that \textit{exact recovery is possible} for $(\Theta_d^+(s_d,a_d))_{d\geq 1}$ in the Coordinate Local case if there exists $Q\in{\Qc}^{CL}_\alpha$ and a selector $\hat{\eta}$ such that
$$
\lim_{d\rightarrow \infty} \sup_{\theta\in\Theta_d^+(s_d,a_d)}\Eb_{Q(P_\theta^{\otimes n})}\vert\hat{\eta}-\eta\vert=0.
$$
We will say that \textit{exact recovery is impossible} for $(\Theta_d^+(s_d,a_d))_{d\geq 1}$ in the Coordinate Local case if
$$
\liminf_{d\rightarrow +\infty} \inf_{Q\in {\Qc}^{CL}_{\alpha}}\inf_{\hat{\eta}=\hat{\eta}(Z^{1},\ldots,Z^{n})\in\Tc}\sup_{\theta\in\Theta_d^+(s,a)}\Eb_{Q(P_\theta^{\otimes n})}\vert\hat{\eta}-\eta\vert >0.
$$
We use similar definitions in the Coordinate Global case with ${\Qc}^{CL}_\alpha$ replaced by $\Qc_\alpha$.

\subsection{Related work}

Variable selection with Hamming loss in the Gaussian mean model in $\Rb^d$ has been studied in the non-private setting in \cite{butucea2018variableSelectionHammingLoss}.
The authors provide non-asymptotic lower and upper bounds on the non-private version of minimax risk \eqref{MinimaxRisk scenario 1}.
As corollaries, they derive necessary and sufficient conditions for almost full recovery and exact recovery to be possible.  If $s,\, d \to \infty$ such that $s/d \to 0$, they highlight a critical value $a^*=(\sigma/\sqrt{n})\sqrt{2\log(d/s-1)}(1+\delta)$ for a specific sequence $\delta = \delta(d,s) \to 0$ such that almost full recovery is possible for $a\geq a^*$ and impossible for $a<a^*$.
Similar results have been obtained for exact recovery with the greater critical value $a^*=(\sigma/\sqrt{n})(\sqrt{2\log(d-s)}+\sqrt{2\log s})$.
In the present paper, we will see how these results are affected by the privacy constraints.

For estimating the $1-$sparse mean of high-dimensional vectors with distribution supported on a compact support it is known that the rates are deteriorates by a factor $d$ under local differential privacy, see \cite{DuchiJordanWainwright2018MinimaxOptimalProcedure}. Under a relaxation of central differential privacy called  $(\alpha,\delta)-$approximate differential privacy -  see for instance \cite{Dwork2006ourData_ourselves}, \cite{Aden_Ali_Ashtiani_Kamath_2020} and \cite{Biswas_Dong_Kamath_Ullman_2020} have provided estimators of the mean and the covariance of high-dimensional Gaussian vectors and theoretical guarantees that do not require additional assumptions on the parameters. In some regimes the rates are not deteriorated and it is therefore difficult to anticipate the role of privacy on each particular problem.

A few papers tackle a slightly different selection problem under privacy constraints mostly under central differential privacy constraints. They are interested in the largest sum of $k$ coordinates of the common mean value $\theta$ of a vector supported on $\{0,1 \}^d$. We are mainly interested in recovering the position of significant coordinates in the $s-$sparse mean vector $\theta$. \\
In \cite{Steinke_Ullman_2017_GDPselection}, the authors study top-$k$ selection under a relaxation of central differential privacy called $({\alpha},\delta)$-approximate differential privacy. However, they use a weighted Hamming loss as described below.
Precisely, if $X_1,\ldots,X_n$ are drawn i.i.d. from some distribution $P$
on $\{0,1\}^d$, they want to find the $k$ greatest coordinates of the mean vector $\theta=\Eb_P[X_1]$ while respecting $({\alpha},\delta)$-differential privacy constraints.
They prove the existence of a $(1,1/(nd))$-differentially private mechanism that outputs $Z\in\{0,1\}^d$ with $k$ non zero coordinates such that
$$
\Eb\left[\sum_{j=1}^d \theta_j {I}(Z_j=1)\right]\geq \max_{\eta\in\{0,1\}^d : \Vert \eta\Vert_1=k}\sum_{j=1}^d\theta_j {I}(\eta_j=1)-\beta
$$
requires $n\gtrsim \sqrt{k}\log d$ samples in the low accuracy regime where $\beta=k/10$.
Moreover, repeated use of the classical exponential mechanism solves this problem with $n=O(\sqrt{k}\log d)$ samples.
In \cite{Bafna_Ullman_2017_PriceGDPselection}, the authors study an empirical version of the problem studied in \cite{Steinke_Ullman_2017_GDPselection}: they want to find the top-k coordinates of the vector $q\in\Rb^d$ defined by $q_j=(1/n)\sum_{i=1}^nX_{i,j}$, $j=1,\ldots,d$ while respecting $({\alpha}, \delta)$-differential privacy constraints.
Let $\tau$ be the $k$-th largest value among the coordinates $\{q_1,\ldots,q_k\}$.
They prove the existence of a $({\alpha},\delta)$-differentially private mechanism that outputs a set $S\subset\llbr 1,d\rrbr$ of $k$ elements such that $q_j\geq \tau-\beta$ for all $j\in S$ requires $n\gtrsim k\log(d)$ samples in the high-accuracy regime where $\beta \asymp \sqrt{\log d/n}$.
In \cite{Ullman_2018_LDPselection}, the author studies the same problem as \cite{Steinke_Ullman_2017_GDPselection} for $k=1$ under non-interactive $\alpha$-local differential privacy constraints. If we consider the low-accuracy regime considered by \cite{Steinke_Ullman_2017_GDPselection}, this result shows that estimating the largest coordinate of a 1-sparse mean $\theta$ under non-interactive $\alpha$-local differential privacy requires $n\gtrsim d\log d/\alpha^2$ samples, which is by a factor $d$ larger than in the central model of $({\alpha},\delta)$-approximate differential privacy.

\subsection{Description of results}

We address the problem of variable selection in a symmetric, strongly log-concave model in $\Rb^d$ under local differential privacy constraints.
We provide lower and upper bounds on the rate of convergence for the expected Hamming loss over classes of at most $s$-sparse vectors whose non-zero coordinates are separated from $0$ by a constant $a>0$.

When we restrict our attention to non-interactive mechanisms that act independently on each coordinate ({\it coordinate local privacy mechanisms}) we have proved that, contrary to the non-private setting, almost full recovery and exact recovery are impossible whatever the value of $a$ in the high-dimensional regime when $n \alpha^2 \lesssim d^2$. This is due to the fact that the loss of information due to privacy may reduce the effective sample size $N:= n \alpha^2/d^2$ under the value 1, and this does not allow support recovery neither exact nor almost full. This result is significantly different from the non-private case where \cite{butucea2018variableSelectionHammingLoss} shows that variable selection is always possible, even for $n=1$ observation for significant enough mean value $a$.\\
However, in the regime $n\alpha^2/d^2\gg \log(d)$ we exhibit a critical value $a^*$ (up to a logarithmic factor) such that exact recovery is possible for all $a\gg a^*$ and impossible for all $a\leq a^*$.
We also prove that these results can be improved when allowing for all non-interactive locally differentially private mechanisms, that we also call {\it coordinate global}. The effective sample size is $Nd$ in this case and it is larger than $N$.\\
Let us note that the separable selectors that we propose are free of the sparsity parameter $s$. They depend on $a$ and methods could be made adaptive to $a$, but this is beyond the scope of this work.\\
For many estimation problems, allowing for sequentially interactive privacy mechanisms, that randomize each vector $X_i$ by using also the publicly available information $Z_1,...,Z_{i-1}$, $i=2,...,n$,  does not improve substantially over non-interactive minimax rates.
This includes for instance density estimation \cite{Butucea_Dubois_Kroll_Saumard}, one-dimensional mean estimation \cite{DuchiJordanWainwright2018MinimaxOptimalProcedure}, and estimation of a linear functional of the true distribution \cite{Rohde_Steinberger_2020_geometrizing}.
However, for some estimation problems (see for instance the estimation of the integrated square of a density, \cite{ButuceaRohdeSteinberger2020quadraticFunctional}) and some testing problems (see \cite{Berrett_Butucea_2020_testing_discrete_distrib} and \cite{ButuceaRohdeSteinberger2020quadraticFunctional}) allowing for sequentially interaction between data-holders can substantially improve over non-interactive minimax rates of estimation or non-interactive minimax rates of testing.
We consider here only non-interactive privacy mechanisms for each vector $X_i$, but we conjecture that the exact and almost full recovery would be improved for interactive privacy mechanisms. It is left for future work to study whether that is indeed the case.\\

The paper is organised as follows. In Section \ref{Section Scenario 1}, we study the minimax risk \eqref{MinimaxRisk scenario 1}.
We first provide a lower bound  which enables us to derive necessary conditions for almost full recovery and exact recovery to be possible in the case where only coordinate local privacy mechanisms are used.
In particular, we prove that almost full recovery is impossible in this case as soon as the quantity $n\alpha^2/d^2$ is bounded from above.
We then provide non-asymptotic upper bounds on the minimax risks in propositions and state more explicit asymptotic sufficient conditions for almost full recovery and exact recovery to be possible in our corollaries.
These conditions and associated results are summarised in Table \ref{Table Exact recovery scenario 1}.
In Section \ref{Section Scenario 2}, we study the minimax risk \eqref{MinimaxRisk scenario 2} and prove that the results of Section \ref{Section Scenario 1} can be improved when any non-interactive (coordinate global) $\alpha$-locally differentially private mechanism is allowed.
See Table \ref{Table Exact recovery scenario 2} for a summary of these results.
Detailed proofs can be found in the Appendix.

\begin{table}
\centering
\begin{tabular}{|l|l|p{4cm}|p{2,8cm}|}
    \hline
        & $a\lesssim \frac{\sigma}{\sqrt{N}}$ & $\frac{\sigma}{\sqrt{N}}\ll a\leq 2\sigma$  & $a\geq 2 \sigma$\\
    \hline
    $N:=\frac{n\alpha^2}{d^2} \lesssim 1$ & impossible & impossible & impossible\\
    \hline
    $N:=\frac{n\alpha^2}{d^2} \gg 1$ & impossible &
    \begin{tabular}{l}
          possible, as soon as \\
          \\
         $a \gg  \frac{\sigma}{\sqrt{N}} \sqrt{ \log(d) } $
         \\
         \\
         if moreover \\
         \\
         $N \gg { \log\left(d \right)} $
    \end{tabular}
    &
    \begin{tabular}{l}
          possible, if \\
           \\
          $\frac{\log(d)}{N} \lesssim 1$
    \end{tabular}
     \\
    \hline
\end{tabular}
\caption{Exact recovery of $\theta$ in either $\Theta^+_d(s,a)$ or $\Theta_d(s,a)$ in the Coordinate Local case.  Similar results hold for almost full recovery with $\log(d)$ replaced by $\log(d/s)$.
}
\label{Table Exact recovery scenario 1}
\end{table}

\begin{table}
\centering
\begin{tabular}{|l|l|p{3,7cm}|p{3,7cm}|}
    \hline
        & $a\lesssim \sigma\sqrt{\frac{\log d}{Nd}}$ & $\sigma\sqrt{\frac{\log d}{Nd}}\ll a\leq 2\sigma$  & $a\geq 2 \sigma$\\
    \hline
    $\frac{Nd}{\log d} \lesssim 1$ & impossible & impossible &
    \begin{tabular}{l}
    impossible if\\
    $a \leq \sigma\sqrt{\log\left(1+\frac{\log d}{16Nd} \right)} $\\
    \\
    \end{tabular}
    \\
    \hline
    $\frac{Nd}{\log d} \gg 1$ &
    impossible
 &
    \begin{tabular}{l}
          possible, as soon as \\
          \\
         $a \gg  \sigma\sqrt{\frac{\log d}{Nd}}  $,
         \\
         \\
         if moreover \\
         \\
         $Nd \gg { \log(d)} $
    \end{tabular}
    &
    \begin{tabular}{l}
          possible
    \end{tabular}
     \\
    \hline
\end{tabular}
\caption{Exact recovery of $\theta$ in either $\Theta^+_d(s,a)$ or $\Theta_d(s,a)$ in the Coordinate Global case. We have set $N=n\alpha^2/d^2$ for a better comparison with the Coordinate Local case. }
\label{Table Exact recovery scenario 2}
\end{table}

\section{Coordinate local non-interactive privacy mechanisms}\label{Section Scenario 1}

In this section, we provide a lower bound on the private minimax risk \eqref{MinimaxRisk scenario 1}.
This enables us to obtain necessary conditions for almost full recovery and exact recovery to be possible in the Coordinate Local scenario.
In particular, we prove that almost full recovery is impossible in the private setting of the Coordinate Local case if the quantity $N:= n\alpha^2/d^2$ is bounded from above.
We then provide upper bounds on the minimax risk that entail sufficient conditions for almost full recovery and exact recovery to be possible.

\subsection{Lower bound}

We first state our lower bound.

\begin{theorem}\label{Thrm Lower Bound Scenario 1}
Assume that the measure $P^{\xi_1}$ of the noise coordinates, is log-concave with a density $p=\exp(-\phi)$, where the potential $\phi$ has a curvature bounded from above by a constant $c_+>0$, that satisfies inequality \eqref{eq:upper_bound_hessian_intro}.
Then for any $a>0$, $\alpha>0$,  $1\leq s\leq d$, $n\geq 1$, we have
\begin{equation}\label{Eq Lower Bound Scenario 1}
  \Rc^{CL}_n(\alpha, \Theta_d^+(s,a)) \geq \left(1-\frac{s}{d}\right)\exp\left( -4n(e^{\alpha/d}-1)^2\min\left\{\frac{c_+ a^2}{4\sigma^2},1 \right\}\right).
\end{equation}
\end{theorem}
The proof of Theorem \ref{Thrm Lower Bound Scenario 1} can be found in Appendix \ref{App Proof Lower bound Scenario 1}.
Some auxiliary results used for the proof of Theorem \ref{Thrm Lower Bound Scenario 1} can be found in Appendix \ref{App Auxialiary Results Lower bound scenario 1}.
Note that since $\Theta_d^+(s,a)\subset \Theta_d(s,a)$ we have $ \Rc^{CL}_n(\alpha, \Theta_d^+(s,a)) \leq  \Rc^{CL}_n(\alpha, \Theta_d(s,a))$, thus the right hand side of \eqref{Eq Lower Bound Scenario 1} is also a lower bound for $\Rc^{CL}_n(\alpha, \Theta_d(s,a))$.

A careful look at the proof of Theorem \ref{Thrm Lower Bound Scenario 1} shows that log-concavity is in fact not needed in the previous result, if we assume the existence of a positive density, converging to zero at infinity, and with a two times continuously differentiable potential achieving \eqref{eq:upper_bound_hessian_intro}.

For better confidentiality in practice, the parameter $\alpha$ must not be too large.
In particular, we assume that $\alpha/d \rightarrow 0$ when $d\rightarrow +\infty$.
We thus have $n(e^{\alpha/d}-1)^2 \sim n\alpha^2/d^2$ and Theorem \ref{Thrm Lower Bound Scenario 1} immediately shows the following.

\begin{corollary}\label{Coro coordinate local Lower bound}
Grant assumptions of Theorem \ref{Prop Upper Bound on the risk for large a Scenario 1}. Let $\alpha>0$,  $1\leq s\leq d$, $n\geq 1$ be such that $s/d\leq C_0$ for some constant $C_0\in (0,1)$, and $\alpha/d\rightarrow 0$ when $d\rightarrow \infty$.
Then, if $n\alpha^2/d^2\leq C_1$ for some constant $C_1>0$ or if $n\alpha^2/d^2\rightarrow \infty$ as $d\rightarrow \infty$ and $a^2\leq C_2 \sigma^2 d^2/n\alpha^2$ for some constant $C_2>0$ depending only on $c_+$, it holds
$$
\Rc^{CL}_n(\alpha, \Theta) \geq C
$$
for some constant $C>0$, where $\Theta=\Theta_d^+(s,a)$ or $\Theta=\Theta_d(s,a)$.
\end{corollary}
Corollary \ref{Coro coordinate local Lower bound} shows that almost full recovery is impossible under local differential privacy constraints if the quantity $n\alpha^2/d^2$ is bounded from above.
In particular, almost full recovery is impossible under local differential privacy constraints in the high-dimensional setting, that is when $n\leq d$, whatever the value of $a$.
Corollary \ref{Coro coordinate local Lower bound} also proves that if $n\alpha^2/d^2 \rightarrow +\infty$ then almost full recovery is impossible if $a\lesssim \sigma d/\sqrt{n\alpha^2}$.

This underlines a strong difference between the private setting and the classical setting, since \cite{butucea2018variableSelectionHammingLoss} proved that in the non-private setting almost full recovery is possible for values of $|a|$ large enough, even if $n=1$. However, both almost full and exact recovery are impossible for any signal value $a$ when the effective size $N=n\alpha^2/d^2 \lesssim 1$  under privacy constraints.

\subsection{Privacy mechanism}\label{Subsection Privacy Mechanism scenario 1}

In this section, we introduce a non-interactive privacy mechanism creating private views $Z^{1},\ldots,Z^{n}$ of the original data $X^{1},\ldots,X^{n}$ that satisfy the local differential privacy constraint of level $\alpha$.
These privatized data will then be used to define a private selector whose risk will be studied in Section \ref{Subsection Upper Bound scenario 1}.

To obtain the privatized data, we first censor the unbounded random variables $X^{i}_j$, for $i=1,\ldots,n$, $j=1,\ldots,d$,  and then make use of an appropriately scaled version of the classical Laplace mechanism.
For all $i\in\llbr 1,n\rrbr$ and $j\in\llbr 1,d\rrbr$ define
\begin{equation}\label{Privacy mechanism}
Z^{i}_j=\sgn[X^{i}_j]+\frac{2 d}{\alpha}W^{i}_j,
\end{equation}
where $\sgn[x]=1$, for $x \geq 0$, and 0, for $x<0$, the $W^{i}_j$'s are i.i.d Laplace$(1)$ random variables,  and $W^{i}_j$ is independent from $X^{i}_j$.

Note that the privacy mechanism defining $(Z^{i})_{i=1,\ldots,n}$ is non-interactive since $Z^{i}$ does only depend on $X^{i}$ and not on $Z^{k}$ for $k\neq i$.
This is also a coordinate local mechanism since $Z^{i}_j$ depends on $X^{i}_j$ but not on the $X^{i}_l$ for $l\neq j$.
The following Proposition shows that it satisfies the condition of $\alpha$-local differential privacy.

\begin{proposition}
For all $i\in\llbr 1,n\rrbr$ and $j\in \llbr 1,d\rrbr$, $Z^{i}_j$ is an $\alpha/d$-differentially private view of $X^{i}_j$.
Consequently, for all $i\in\llbr 1,n\rrbr$ $Z^{i}=(Z^{i}_j)_{j=1,\ldots,d}$ is an $\alpha$-differentially private view of $X^{i}$.
\end{proposition}

\begin{proof}
Set $r=2 d/\alpha$.
By definition of the privacy mechanism \eqref{Privacy mechanism}, the conditional density of $Z^{i}_j$ given $X^{i}_j=x$ can be written as
$$
q^{Z^{i}_j\mid X^{i}_j=x}(z)= \frac{1}{2r}\exp\left(-\frac{\vert z-\sgn[x]\vert}{r} \right).
$$
Thus, by the reverse and the ordinary triangle inequality it holds for all $i\in \llbr 1,n \rrbr$, $j\in\llbr 1,d\rrbr$ and all $x,x',z\in\Rb$,
\begin{align*}
\frac{q^{Z^{i}_j\mid X^{i}_j=x}(z)}{q^{Z^{i}_j\mid X^{i}_j=x'}(z)}&=\exp\left(\frac{\vert z-\sgn[x']\vert}{r} - \frac{\vert z- \sgn[x]\vert}{r} \right)\\
&\leq \exp\left(\frac{\vert \sgn[x']- \sgn[x] \vert}{r}  \right)\\
&\leq \exp\left(\frac{2}{r}  \right)
\leq \exp\left(\frac{\alpha}{d}  \right).
\end{align*}
This proves that $Z^{i}_j$ is an $\alpha/d$-differentially private view of $X^{i}_j$.
Let us check that $Z^{i}$ is an $\alpha$-differentially private view of $X^{i}$. Denote by $q^{Z^{i}\mid X^{i}=x}$ the conditional density of $Z^{i}$ given $X^{i}=x$ and note that for all $x,x',z\in\Rb^d$ it holds
$$
\frac{q^{Z^{i}\mid X^{i}=x}(z)}{q^{Z^{i}\mid X^{i}=x'}(z)}=\prod_{j=1}^d\frac{q^{Z^{i}_j\mid X^{i}_j=x_j}(z_j)}{q^{Z^{i}_j\mid X^{i}_j=x_j'}(z_j)}\leq e^\alpha,
$$
using the independence of the coordinates $X_1^i,...,X_d^i$ and the conditional independence of $Z_1^i,...,Z_d^i$ given $X^i$.
\end{proof}

\subsection{Upper bounds}\label{Subsection Upper Bound scenario 1}

Using these privatized data, we define two selectors that will provide upper bounds on the minimax risk \eqref{MinimaxRisk scenario 1}.
For the class $\Theta_d^+(s,a)$, we will use the selector $\hat{\eta}^+$ with the components
\begin{equation}\label{Selecteur1}
\hat{\eta}^+_j=I\left(\frac{1}{n} \sum_{i=1}^n Z^{i}_j\geq \tau\right), \quad j=1,\ldots,d,
\end{equation}
where the threshold $\tau$ has to be properly chosen, later on.
For the class $\Theta_d(s,a)$, we will use the selector $\hat{\eta}$ with the components
\begin{equation}\label{Selecteur2}
\hat{\eta}_j=I\left(\left \vert\frac{1}{n} \sum_{i=1}^n Z^{i}_j \right\vert\geq \tau\right), \quad j=1,\ldots,d,
\end{equation}
where $\tau$ to be defined later on.
Note that $\hat{\eta}^+$ and $\hat{\eta}$ are separable selectors since $\hat{\eta}^+_j$ and $\hat{\eta}_j$ depend only on $(Z^{i}_j)_{i=1,\ldots,n}$ and not on the $Z^{i}_k$ for $k\neq j$.
We now study the performances of these selectors. Recall that $\Phi$ is c.d.f. of the normal distribution.


\begin{proposition}\label{Prop Upper Bound on the risk for large a Scenario 1}
Assume that $a \geq 2 \sigma$. Set $C_1:= 2\Phi( 2 \sqrt{c} )-1>0$.
If $\tau$ is chosen such that
$$  C_1-\tau > 0, \, \tau\alpha/(8d)\leq 1 \text{ and } \alpha(C_1-\tau)/(8d)\leq1,
$$
then it holds for all $\theta\in\Theta_d^+(s,a)$,
\begin{multline}\label{Eq Upper bound large a scenario 1 theta+}
\Eb\left[\frac{1}{s}\vert\hat{\eta}^+-\eta\vert \right] 
\leq \frac{d-\vert S\vert}{s}\left[\exp\left(-\frac{n\tau^2}{2^3 }\right)
+\exp\left(-\frac{\tau^2n\alpha^2}{2^7 d^2}  \right)\right]\\
 +\frac{\vert S\vert}{s}\left[\exp\left(-\frac{n(C_1-\tau)^2}{2^3 }\right)
 +\exp\left(-\frac{(C_1-\tau)^2n\alpha^2}{2^7 d^2}  \right) \right],
\end{multline}
and for all $\theta\in\Theta_d(s,a)$ it holds
\begin{multline}\label{Eq Upper bound large a scenario 1 theta}
\Eb\left[\frac{1}{s}\vert\hat{\eta}-\eta\vert \right]\leq 2\frac{d-\vert S\vert}{s}\left[\exp\left(-\frac{n\tau^2}{2^3 }\right)
+\exp\left(-\frac{\tau^2n\alpha^2}{2^7 d^2}  \right)\right]\\
 + 2\frac{\vert S\vert}{s}\left[\exp\left(-\frac{n(C_1-\tau)^2}{2^3 }\right) +\exp\left(-\frac{(C_1-\tau)^2n\alpha^2}{2^7 d^2}  \right) \right],
\end{multline}
where $S$ denotes the support of $\theta$.
\end{proposition}

The proof of Proposition \ref{Prop Upper Bound on the risk for large a Scenario 1} is given in Section \ref{App proof Upper Bound Large a scenario 1} in the Appendix.
Some auxiliary results used in the proof of Proposition \ref{Prop Upper Bound on the risk for large a Scenario 1} can be found in Appendix \ref{App Auxiliary results upper bounds scenario 1}.
The following Corollary gives sufficient conditions so that almost full recovery and exact recovery are possible under local differential privacy in the Coordinate Local case when $a\geq 2\sigma$.

\begin{corollary}\label{Corollary Sufficient Conditions for AFR large a scenario 1}
Set $C_1=2\Phi(2 \sqrt{c})-1>0$. Assume  that
$$
\alpha/d\rightarrow 0, n\alpha^2/d^2\rightarrow +\infty \text{ and } \limsup \frac{\log(d/s)}{n\alpha^2/d^2}<\frac{C_1^2}{2^9}.
$$
Then the selector $\hat{\eta}^+$ defined by \eqref{Selecteur1} with $\tau=C_1/2$ satisfies
\begin{equation}\label{Eq AFR large a scenario 1}
  \sup_{\theta\in\Theta}\frac{1}{s}\Eb_{Q(P_\theta^{\otimes n})}\vert\hat{\eta}^+(Z^{1},\ldots,Z^{d})-\eta\vert\rightarrow 0,
\end{equation}
for all $a\geq 2\sigma$, where $\Theta=\Theta_d^+(s,a)$ or $\Theta=\Theta_d(s,a)$.
If, in addition, $ \limsup \frac{\log(d)}{n\alpha^2/d^2}<\frac{C_1^2}{2^9}$, then
\begin{equation}\label{Eq ER large a scenario 1}
    \sup_{\theta\in\Theta}\Eb_{Q(P_\theta^{\otimes n})}\vert\hat{\eta}^+(Z^{1},\ldots,Z^{d})-\eta\vert\rightarrow 0,
\end{equation}
for all $a\geq 2\sigma$.
\end{corollary}

The proof of Corollary \ref{Corollary Sufficient Conditions for AFR large a scenario 1} is given in Section \ref{App Corollary Upper Bound Large a scenario 1} in the Appendix.
Since we have seen that almost full recovery is impossible when $n\alpha^2/d^2$ is bounded from above or when $n\alpha^2/d^2 \rightarrow +\infty$ and $a\lesssim (\sigma d)/(\sqrt{n}\alpha)$, it remains to study the case where $n\alpha^2/d^2 \rightarrow +\infty$ and $\sigma d/(\sqrt{n}\alpha)\ll a\leq 2 \sigma$.
This is done below.

\begin{proposition}\label{Prop Upper Bound for small a scenario 1}
 Let $a>0$. If $\tau$ is chosen such that
 $\tau < 2 a/\sigma p(2)$, $\tau\alpha/(8d)<1$ and $\alpha(a/\sigma p(2)-\tau/2)/(4d)\leq 1$
 then it holds for all $\theta\in\Theta_d^+(s,a)$,
\begin{align*}
\Eb\left[\frac{1}{s}\vert\hat{\eta}^+-\eta\vert \right]&\leq \frac{d-\vert S\vert}{s}\left[\exp\left(-\frac{n\tau^2}{2^3}\right)+\exp\left(-\frac{\tau^2 n\alpha^2}{2^7 d^2}  \right)\right]\\
& +\frac{\vert S\vert}{s}\left[\exp\left(-\frac{n(a/\sigma p(2)-\tau/2)^2}{2^3 }\right) +\exp\left(-\frac{(a/\sigma p(2)-\tau/2)^2 n\alpha^2}{2^5 d^2}  \right) \right],
\end{align*}
where $S$ denotes the support of $\theta$.
\end{proposition}

The proof of Proposition \ref{Prop Upper Bound for small a scenario 1} can be found in Section \ref{App Proof Upper Bound Small a scenario 1} in the Appendix.
Note that as for the case $a\geq 2\sigma$, if $\theta\in\Theta_d(s,a)$ we use $\hat{\eta}$ instead of $\hat{\eta}^+$ and we can prove the same result with  an  extra multiplicative factor 2.
The next corollary gives new sufficient conditions so that almost full recovery and exact recovery are possible.

\begin{corollary}\label{Corollary Sufficient Conditions for AFR small a}
Assume that $\alpha/d\rightarrow 0$, $n\alpha^2/d^2\rightarrow +\infty$ and $\sigma d/(\sqrt{n}\alpha)\ll a\leq 2 \sigma$.
The selector $\hat{\eta}^+$ defined by \eqref{Selecteur1} with  $\tau= p(2)a /\sigma $ satisfies for $d$ large enough
$$
\sup_{\theta\in\Theta_d^+(s,a)}\frac{1}{s}\Eb_{Q(P_\theta^{\otimes n})}\vert\hat{\eta}^+(Z^{1},\ldots,Z^{d})-\eta\vert\leq 2\exp\left(\log\left(\frac{d}{s} \right)-\frac{p^2(2) a^2n\alpha^2}{2^{9}\sigma^2 d^2} \right).
$$
In particular, if $a \gg \frac{\sigma d}{ \alpha \sqrt{n}}  \log^{1/2}\left(\frac{d}{s} \right)$ it holds
\begin{equation}\label{Eq AFR small a scenario 1}
  \sup_{\theta\in\Theta_d^+(s,a)}\frac{1}{s}\Eb_{Q(P_\theta^{\otimes n})}\vert\hat{\eta}^+(Z^{1},\ldots,Z^{d})-\eta\vert\rightarrow 0.
\end{equation}
Moreover, if $a \gg \frac{\sigma d}{ \alpha \sqrt{n}} \log^{1/2}\left(d \right)$ then
\begin{equation}\label{Eq ER small a scenario 1}
  \sup_{\theta\in\Theta_d^+(s,a)}\Eb_{Q(P_\theta^{\otimes n})}\vert\hat{\eta}^+(Z^{1},\ldots,Z^{d}))-\eta\vert\rightarrow 0 .
\end{equation}
\end{corollary}

If $n\alpha^2/d^2\rightarrow \infty$ with $(n\alpha^2/d^2) \gg \log(d/s)$, then Corollary \ref{Corollary Sufficient Conditions for AFR small a} combined with Corollary \ref{Corollary Sufficient Conditions for AFR large a scenario 1} and with the lower bound \eqref{Eq Lower Bound Scenario 1} prove a phase transition result (up to $\log$ factors) at the value $a^*=a^*(n,\alpha,d,\sigma)=\sigma d/(\alpha \sqrt{n})$. Indeed, we get that almost full recovery is impossible in the Coordinate Local case for all $a\leq C a^*$ and is possible for all $a\gg a^* \log^{1/2}(d/s)$.

\section{Coordinate global non-interactive privacy mechanisms}\label{Section Scenario 2}

In this section, we study the minimax risk \eqref{MinimaxRisk scenario 2}.
We prove that in the Coordinate Global case, almost full recovery and exact recovery are possible under weaker assumptions than the one we obtained for the Coordinate Local case.

\subsection{Privacy mechanism}

We describe in this section the privacy mechanism we use to obtain private data that will be used to design a private selector and to obtain upper bounds on the minimax risk \eqref{MinimaxRisk scenario 2} in the Coordinate Global case.

For all $i\in\llbr 1,n\rrbr$, the private view $Z^{i}$ of $X^{i}$ is obtained using the following steps:
\begin{itemize}
\item Compute $f(X^{i})=(\sgn[X^{i}_j])_{j=1,\ldots,d}$. For short, let us denote $\tilde{X}^{i}=f(X^i)$.
\item Sample $Y^{i}\sim \Bc(\pi_\alpha)$ where $\pi_\alpha=e^\alpha/(e^\alpha+1)$ and generate $\tilde{Z}^{i}$ uniformly distributed on the set
$$
\left\{ \tilde{z}\in\left\{-B,B\right\}^d \, \mid \, \langle \tilde{z}, \tilde{X}^{i} \rangle> 0 \text{ or } (\langle \tilde{z}, \tilde{X}^{i} \rangle=0 \text{ and }\tilde{z}_1=B\tilde{X}^{i}_1)  \right\} $$
 if  $Y^{i}=1$, respectively on the set
$$
\left\{ \tilde{z}\in\left\{-B,B\right\}^d \, \mid \, \langle \tilde{z}, \tilde{X}^{i} \rangle< 0 \text{ or } (\langle \tilde{z}, \tilde{X}^{i} \rangle=0 \text{ and }\tilde{z}_1=-B\tilde{X}^{i}_1)  \right\}
$$
if $Y^{i}=0$, with
\begin{equation}
    \label{Kd}
B=\frac{e^\alpha+1}{e^\alpha-1}K_d, \text{ where } \frac{1}{K_d}=
\begin{cases}
\frac{1}{2^{d-1}}\binom{d-1}{\frac{d-1}{2}} & \text{ if $d$ is odd}\\
\frac{(d-2)!(d-2)}{2^{d-1}(\frac{d}{2}-1)!\frac{d}{2}!} & \text{ if $d$ is even}.
\end{cases}
\end{equation}
\item Define the vector $Z^{i}$ by $Z^{i}=\tilde{Z}^{i}$ if $d$ is odd, and by its components
$$
Z^{i}_j=
\begin{cases}
\frac{d-2}{2(d-1)}\tilde{Z}^{i}_1 & \text{if } j=1\\
\tilde{Z}^{i}_j & \forall j\in\llbr 2,d \rrbr,
\end{cases}
$$
if $d$ is even.
\end{itemize}
This mechanism is strongly inspired by the one proposed by Duchi et al. \cite{DuchiJordanWainwright2018MinimaxOptimalProcedure} for mean estimation  on the set of distributions $P$ supported on $\Bb_\infty(r)\subset \Rb^d$ with $\Vert\Eb[X]\Vert_0\leq s$.
In particular, if $d$ is odd, the event $\{\langle \tilde{z}, \tilde{X}^{i} \rangle=0 \}$ has probability zero for all $\tilde{z}\in\{-B,B\}^d$ and our mechanism coincides in this case  with the one proposed by Duchi et al. \cite{DuchiJordanWainwright2018MinimaxOptimalProcedure} applied to $\sgn(X^i)$ instead of $X^i$. 


\begin{proposition}\label{Prop alpha DP scenario 2}
For all $i\in\llbr 1,n\rrbr$, $Z^{i}$ is an $\alpha$-differentially private view of $X^{i}$.
\end{proposition}

The following proposition will be useful in the analysis of the selector proposed in Section \ref{Subsection Upper bound scenario 2}.

\begin{proposition}\label{Prop Expectation of the private data scenario 2}
For all $i\in\llbr 1,n\rrbr$, it holds
$$
\Eb[Z^{i}\mid X^{i}]=f(X^{i}).
$$
\end{proposition}

The proofs of Proposition \ref{Prop alpha DP scenario 2} and Proposition \ref{Prop Expectation of the private data scenario 2} can be found respectively in Section \ref{App Proof alpha DP scenario 2} and \ref{App Proof expectation of the private data scenario 2} of the Appendix.
Note that it also holds $\Eb[Z^{i}\mid X^{i}]=f(X_i)$ when $Z_i$ is produced via the Laplace mechanism described in Subsection \ref{Subsection Privacy Mechanism scenario 1}.
However the variance $\Var(Z^{i}_j\mid X^{i})$ is slower by a multiplicative factor $d$ when $Z^{i}$ is produced with the Laplace mechanism than when it is obtained with the above coordinate global mechanism.
Indeed,  if $Z^{i}$ is produced with the above mechanism, then we have $\Var(Z^{i}_j\mid X^{i})\leq B^2$. Stirling's approximation yields $K_d^2\lesssim d$ for $d$ large enough, see Lemma \ref{Lem Asymptotic Property Kd} in Appendix \ref{App Proof Asymptotic Property Kd} for details. Thus, if $\alpha$ is bounded, we obtain $\Var(Z^{i}_j\mid X^{i})\leq d/\alpha^2$.
Now, if $Z^{i}$ is produced with the Laplace mechanism then it holds $\Var(Z^{i}_j\mid X^{i})=8d^2/\alpha^2$. This faster variance  explains that we will obtain better results when allowing for coordinate global mechanisms.

\subsection{Upper bounds}\label{Subsection Upper bound scenario 2}

Using the privatized data of the previous subsection, we define two selectors that will enable us to obtain upper bounds on the minimax risk \eqref{MinimaxRisk scenario 2}.
As in the Coordinate Local case, for the class $\Theta_d^+(s,a)$, we will use the selector $\hat{\eta}^+$ with the components
\begin{equation}\label{Selecteur3}
\hat{\eta}^+_j=I\left(\frac{1}{n} \sum_{i=1}^n Z^{i}_j\geq \tau\right), \quad j=1,\ldots,d,
\end{equation}
where the threshold $\tau$ has to be chosen.
For the class $\Theta_d(s,a)$, we will use the selector $\hat{\eta}$ with the components
\begin{equation}\label{Selecteur4}
\hat{\eta}_j=I\left(\left \vert\frac{1}{n} \sum_{i=1}^n Z^{i}_j \right\vert\geq \tau\right), \quad j=1,\ldots,d.
\end{equation}
We now study the performances of these selectors.

The following result gives an upper bound on the risk of  selector \eqref{Selecteur3}  when $a\geq C\sigma$ and will enable us to obtain sufficient conditions so that almost full recovery is possible when $a\geq 2\sigma$ in the Coordinate Global case.

\begin{proposition}\label{Prop Upper Bound Large a scenario 2}
Assume that $a> 2 \sigma$ and set $C_1:= 2\Phi(2\sqrt{c})-1>0$. 
If $\tau$ is chosen such that
$
 C_1-\tau> 0,
$
then it holds for all $\theta\in\Theta_d^+(s,a)$,
$$
\Eb\left[\frac{1}{s}\vert\hat{\eta}^+-\eta\vert \right]\leq \frac{d-\vert S\vert}{s}\exp\left(-\frac{n\tau^2}{2B^2 }\right)+\frac{\vert S\vert}{s}\exp\left(-\frac{n(C_1-\tau)^2}{2B^2}\right) ,
$$
where $S$ denotes the support of $\theta$.
In particular, choosing $\tau=C_1/2$ yields
$$
\sup_{\theta\in\Theta_d^+(s,a)} \Eb\left[\frac{1}{s}\vert\hat{\eta}^+ -\eta\vert \right]\leq \frac{d}{s}\exp\left( -\frac{C_1^2n(e^\alpha-1)^2}{8(e^\alpha+1)^2 K_d^2} \right)
$$
for all $a\geq 2\sigma$.
\end{proposition}
The proof of Proposition \ref{Prop Upper Bound Large a scenario 2} can be found in section \ref{App Proof upper bound large a scenario 2} of the Appendix.
Note that we can provide similar results on the class  $\Theta_d(s,a)$ considering the selector $\hat{\eta}$.
The  upper bounds are the same as for the class $\Theta_d^+(s,a)$ up to a multiplicative factor 2 that comes from the use in the proof of the two-sided Hoeffding's inequality instead of the one-sided inequality.
Since $K_d\leq C\sqrt d$ for $d$ large enough,  we obtain that a sufficient condition for almost full recovery to be possible when $a\geq 2\sigma$ is that $\frac{n(e^\alpha-1)^2}{(e^\alpha+1)^2 d} \gtrsim \log(d/s)$.
Moreover, using that $(e^\alpha-1)^2/(e^\alpha+1)^2\geq 0.2\alpha^2$ if $\alpha\leq 1$, we obtain that a sufficient condition for almost full recovery to be possible when $a\geq 2\sigma$ and $\alpha\leq 1$ is that $n\alpha^2/d \gtrsim \log(d/s)$.
This improves the result we obtained when we considered only privacy mechanisms acting coordinates by coordinates for which we needed $n\alpha^2/d^2\gtrsim \log(d/s)$.
We now deal with the case $a \ll \sigma$.
\begin{proposition}\label{Prop Upper Bound Small a scenario 2}
Let $a>0$ and $a\leq 2\sigma$. If $\tau$ is chosen such that $\tau< 2p(2)a/\sigma,
$
then it holds for all $\theta\in\Theta_d^+(s,a)$,
$$
\Eb\left[\frac{1}{s}\vert\hat{\eta}^+-\eta\vert \right]\leq \frac{d-\vert S\vert}{s}\exp\left(-\frac{n\tau^2}{2B^2 }\right)+\frac{\vert S\vert}{s}\exp\left(-\frac{n(2p(2)a/\sigma-\tau)^2}{2B^2}\right) ,
$$
where $S$ denotes the support of $\theta$.
\end{proposition}

The proof of Proposition \ref{Prop Upper Bound Small a scenario 2} can be found in Appendix \ref{App Proof upper bound small a scenario 2}.

\begin{corollary}\label{Corollary Sufficient Conditions for AFR small a, scenario 2}
Assume that $\alpha/d\rightarrow 0$, $n\alpha^2/d\rightarrow +\infty$ and $\sigma \sqrt{d}/(\alpha\sqrt{n})\ll a\leq 2 \sigma$.
The selector $\hat{\eta}^+$ defined by \eqref{Selecteur3} with $\tau=p(2)a/\sigma$ satisfies for $n,d$ large enough
$$
\sup_{\theta\in\Theta_d^+(s,a)}\frac{1}{s} \Eb_{Q(P_\theta^{\otimes n})} \vert\hat{\eta}^+(Z^{1},\ldots,Z^{d}))-\eta\vert
\leq \frac{d}{s}\exp\left(- \frac{ n(e^\alpha-1)^2p^2(2)a^2}{2\sigma^2(e^\alpha+1)^2K_d^2} \right).$$
In particular, if $\alpha\in(0,1]$, if $n\alpha^2/d\rightarrow +\infty$ with $n\alpha^2/d\gg \log(d) $ then it holds
\begin{equation*}\label{Eq ER small a scenario 2}
  \sup_{\theta\in\Theta_d^+(s,a)} \frac 1s \Eb_{Q(P_\theta^{\otimes n})} \vert\hat{\eta}^+(Z^{1},\ldots,Z^{d}))-\eta\vert
  \rightarrow 0,
\end{equation*}
for all $a$ satisfying $\sigma\sqrt{\frac{d}{n\alpha^2}}\sqrt{\log(d/s) }\ll a\leq 2\sigma$; and also
\begin{equation*}\label{Eq AFR small a scenario 2}
  \sup_{\theta\in\Theta_d^+(s,a)}\Eb_{Q(P_\theta^{\otimes n})}\vert\hat{\eta}^+(Z^{1},\ldots,Z^{d}))-\eta\vert\rightarrow 0,
\end{equation*}
for all $a$ satisfying $\sigma\sqrt{\frac{d}{n\alpha^2}}\sqrt{\log(d) }\ll a\leq 2\sigma$.
\end{corollary}

The first statement in Corollary \ref{Corollary Sufficient Conditions for AFR small a, scenario 2} is a direct consequence of Proposition \ref{Prop Upper Bound Small a scenario 2}.
The second statement is a direct consequence of the first one where we have used $(e^\alpha-1)^2/(e^\alpha+1)^2\geq 0.2\alpha^2$ for $\alpha\in(0,1]$ and $K_d\leq C\sqrt{d}$ for $d$ large enough.
In the next subsection, we complement these results with a lower bound.
This will enable us to exhibit a value $a^*$ such that exact recovery is impossible for all $a\leq a^*$ and possible for $a\gg a^*$ under the assumptions $\alpha\in(0,1]$ and  $n\alpha^2/d\rightarrow \infty$ with $n\alpha^2/d\gg \log(d)$.

\subsection{Lower bound}
Recall that $P_0$ denotes the distribution of $\sigma\xi_1$ and $P_a$ that of $a+ \sigma \xi_1$ and denote by $\chi^2(P_0,P_a)$ the chi-square discrepancy between the two distributions.
\begin{proposition}\label{Prop Lower Bound scenario 2}
For any $a>0$ such that $\chi^2(P_0,P_a) < \infty$, $\alpha>0$, $d\geq 4$, $1\leq s\leq d$, $n\geq 1$, we have

$$
\inf_{Q\in \Qc_{\alpha}}\inf_{\hat{\eta}\in\Tc}\sup_{\theta\in\Theta_d^+(s,a)}\Eb_{Q(P_\theta^{\otimes n})}\vert\hat{\eta}-\eta\vert\geq \frac{1}{4}\left(1- \frac{2n(e^\alpha-1)^2}{d\log(d)} \chi^2(P_0,P_a) \right).
$$
Assume now that the measure $P^{\xi_1}$ of the noise coordinates is strongly log-concave, with density $p=\exp(-\phi)$, with a potential $\phi$ that has a curvature bounded from above by a constant $c_+$ as in \eqref{eq:upper_bound_hessian_intro}. Then
\begin{equation}\label{eq:upper_bound_chi2}
    \chi^{2}\left(P_{0},P_{a}\right)\leq\exp\left(c_{+}\left(\frac{a}{\sigma}\right)^{2}\right)-1.
\end{equation}


\end{proposition}
Note that Inequality \eqref{eq:upper_bound_chi2} is sharp in the sense that in the Gaussian case, $c_+=1$ holds and Inequality \eqref{eq:lem_chi2} turns out to be an equality. Note also that log-concavity is actually not needed in Proposition \ref{Prop Lower Bound scenario 2}, since we only require an upper bound on the curvature of the potential $\phi$.

The proof of Proposition \ref{Prop Lower Bound scenario 2} is based on a private version of Fano's method, see Proposition~2 in \cite{DuchiJordanWainwright2018MinimaxOptimalProcedure}.
It can be found in Section \ref{App Proof Lower bound scenario 2} of the Appendix.
Using that $(e^\alpha-1)^2\leq 4\alpha^2$ for $\alpha\in(0,1)$ and $\exp(c_{+}x^2)-1\leq Lx^2$ for $0\leq x\leq 2$ and some constant $L$ only depending on $c_+$ (e.g. $L=(\exp(2c_+)-1)/2$), Proposition \ref{Prop Lower Bound scenario 2} immediately shows the following.
\begin{corollary}
Let $\alpha\in(0,1)$. If $n\alpha^2/(d\log d)\leq C/(32L)$ for some constant $C\in(0,1)$ then it holds
$$
\inf_{Q\in \Qc_{\alpha}}\inf_{\hat{\eta}\in\Tc}\sup_{\theta\in\Theta_d^+(s,a)}\Eb_{Q(P_\theta^{\otimes n})}\vert\hat{\eta}-\eta\vert\geq \frac{1}{4}\left(1-C\right)>0,
$$
for all $a\leq 2\sigma$.
\end{corollary}
This shows that exact recovery is impossible for all $a\leq 2\sigma$ if $n\alpha^2/(d\log d)\leq C/(32L)$ for some constant $C\in(0,1)$.
Proposition \ref{Prop Lower Bound scenario 2} also implies that exact recovery is impossible if $a\leq \sigma \sqrt{\log(1+Cd\log d/(8n\alpha^2))/c_+}$ for some constant $C\in(0,1)$. However, unlike the coordinate local case, the lower bound provided by Proposition \ref{Prop Lower Bound scenario 2} does not allow us to say that exact recovery is also impossible for $a\geq \max\{2\sigma,\sigma \sqrt{\log(1+Cd\log d/(8n\alpha^2))/c_+}\} $ when $n\alpha^2/(d\log d)$ is bounded from above.
The following corollary is also a direct consequence of Proposition \ref{Prop Lower Bound scenario 2}. It shows that when $n\alpha^2/(d\log d)\rightarrow \infty$, exact recovery is still impossible if $a$ is too small.
\begin{corollary}
If $\alpha\in(0,1)$, $n\alpha^2/d\rightarrow +\infty$ with $n\alpha^2/d\gg \log d$ and $a\leq (\sigma/(16L))\sqrt{d\log d/(n\alpha^2)} $ it holds
$$
\liminf_{d\rightarrow +\infty}\inf_{Q\in \Qc_{\alpha}}\inf_{\hat{\eta}\in\Tc}\sup_{\theta\in\Theta_d^+(s,a)}\Eb_{Q(P_\theta^{\otimes n})}\vert\hat{\eta}-\eta\vert\geq \frac{1}{8}.
$$
\end{corollary}
The lower bound of Proposition \ref{Prop Lower Bound scenario 2} combined with the upper bounds of Subsection \ref{Subsection Upper bound scenario 2}  exhibit a phase transition at the value $a^*$ (up to a logarithmic factor) such that exact recovery is impossible for all $a\leq a^*$ and possible for $a\gg a^*$ under the assumptions $\alpha\in(0,1]$ and  $n\alpha^2/d\rightarrow \infty$ with $ n\alpha^2/d 
\gg \log(d)$.
Precisely, set
$$
a^*=a^*(n,\alpha,d,\sigma)=\frac{\sigma}{16L}\sqrt{\frac{d\log d}{n\alpha^2}},
$$
where we recall that $L=(\exp(2c_+)-1)/2$. Proposition \ref{Prop Lower Bound scenario 2} combined with Corollary \ref{Corollary Sufficient Conditions for AFR small a, scenario 2} and Proposition \ref{Prop Upper Bound Large a scenario 2} give the following result.
\begin{corollary}
Assume that $\alpha\in(0,1]$ and $n\alpha^2/d\rightarrow +\infty$ with $ n\alpha^2/d 
\gg \log(d) $.
Then, exact recovery is impossible for all $a\leq a^*$ and is possible for all $a\gg a^*.$
\end{corollary}

Note that Proposition \ref{Prop Lower Bound scenario 2} does not allow us to obtain impossibility results for almost full recovery in the regime $n\alpha^2/(d\log d)\gg 1$.
Its proof relies on  a private Fano's method  (Proposition 2 in \cite{DuchiJordanWainwright2018MinimaxOptimalProcedure}) applied with the family of distributions $\left\{\Nc(a\omega_i,\sigma^2 I_d), i=1,\ldots,d\right\}$ where $\omega_i\in \{0,1\}^d$ is defined by $\omega_{ij}=\delta_{ij}$ and $\delta$ is the Kronecker delta.
The same proof with $\omega_i$ defined by $\omega_{ij}=1$ if $j\in\llbr (i-1)s+1, is\rrbr$ and $\omega_{ij}=0$ otherwise for $i=1,\ldots \lfloor d/s\rfloor$, provides the following lower bound.
\begin{proposition}\label{Prop Lower bound AFR scenario 2}
For any $a>0$ such that $\chi^2(P_0, P_a) < \infty$, $\alpha>0$, $n\geq 1$. If $d/s\leq 4$ then we have
$$
\inf_{Q\in \Qc_{\alpha}}\inf_{\hat{\eta}\in\Tc}\sup_{\theta\in\Theta_d^+(s,a)}\frac{1}{s}\Eb_{Q(P_\theta^{\otimes n})}\vert\hat{\eta}-\eta\vert\geq \frac{1}{4}\left(1- \frac{2n(e^\alpha-1)^2}{\lfloor d/s\rfloor \log(\lfloor d/s\rfloor)} \chi^2(P_0^{\otimes s},P_a^{\otimes s})
\right).
$$
If the noise has a potential $\phi$ that is two times continuously differentiable, with curvature bounded from above by a constant $c_+$ as in \eqref{eq:upper_bound_hessian_intro}, then it holds $\chi^2( P_0^{\otimes s},P_a^{\otimes s} )=\exp(s \cdot c_+a^2/\sigma^2) - 1$.
\end{proposition}
Note that $\chi^2(P_0^{\otimes s},P_a^{\otimes s}) = (\chi^2(P_0,P_a)+1)^s - 1$.

However, this bound turns out to be suboptimal in the sense that when it holds $ n\alpha^2/d 
\gg \log(d/s) $ the combination of this bound with upper bounds in Proposition \ref{Prop Upper Bound Large a scenario 2} and Corollary \ref{Corollary Sufficient Conditions for AFR small a, scenario 2} allows us to exhibit the critical value $a^*$ for almost full recovery only up to a logarithmic factor times the sparsity $s$.
Indeed, on the one hand Proposition \ref{Prop Upper Bound Large a scenario 2} and Corollary \ref{Corollary Sufficient Conditions for AFR small a, scenario 2} prove that almost full recovery is possible for all $a\gg \sigma\sqrt{d/(n\alpha^2)}\sqrt{
\log(d/s) }$ in the regime  $ n\alpha^2/d \gg \log(d/s) $.
On the other hand Proposition \ref{Prop Lower bound AFR scenario 2} proves that, in the same regime, almost full recovery is impossible for $a\lesssim (\sigma/s)\sqrt{d/(n\alpha^2)}\sqrt{\log(d/s) }$ but does not allow us to say what happens for $(\sigma/s)\sqrt{d/(n\alpha^2)}\sqrt{\log(d/s) }\ll a \lesssim\sigma\sqrt{d/(n\alpha^2)}\sqrt{\log(d/s) }$.

\begin{table}
\centering
\begin{tabular}{|l|l|p{4cm}|l|}
    \hline
        & $a\lesssim \frac{\sigma}{s} \sqrt{\frac{\log (d/s)}{Nd}}$
        &  $\frac{\sigma}{s} \sqrt{\frac{\log (d/s)}{Nd}} \ll a \lesssim \frac{\sigma}{\sqrt{s}}$ & $a\gg \frac \sigma{\sqrt{s}}$\\
    \hline
   $s \frac{Nd}{\log (d/s)} \lesssim 1$ & impossible & impossible &  ? \\
   \hline
& & &\\
    $s \frac{Nd}{\log (d/s)} \gg 1$ &
    impossible &
    \begin{tabular}{l}
         possible, as soon as \\
         \\
         $a \gg \sigma \sqrt{\frac{\log (d/s)}{Nd}} $ \\
        \\
        if moreover\\
          \\
          $Nd \gg \log(d/s)$
    \end{tabular}
 & possible \\
          \hline
\end{tabular}
\caption{Almost full recovery of $\theta$ in either $\Theta^+_d(s,a)$ or $\Theta_d(s,a)$ in the Coordinate Global case. We have set $N=n\alpha^2/d^2$. }
\label{Table Almost full recovery scenario 2}
\end{table}


\appendix

\section{Proofs of Section \ref{Section Scenario 1}}

\subsection{Some auxiliary results for the proof of the lower bound}\label{App Auxialiary Results Lower bound scenario 1}

The proof of Theorem \ref{Thrm Lower Bound Scenario 1} strongly relies on the following result known as the Bayesian version of the Neyman-Pearson lemma.

\begin{theorem}[\cite{LehmannRomano2006testingStatisticalHypotheses}, Problem 3.10]\label{BayesianNeaymanPearson}
Let $P_0$ and $P_1$ be probability distributions possessing densities $p_0$ and $p_1$ with respect to a measure $\mu$.
Consider the problem of testing $H_0:P=P_0$ against $H_1 : P=P_1$, and suppose that known probabilities $\pi$ and $1-\pi$ can be assigned to $H_0$ and $H_1$ prior to the experiment.
Then the test $T^*$ given by
$$
T^*(X)=I((1-\pi) p_1(X)>\pi p_0(X))
$$
is a minimizer of the overall probability of error resulting from the use of a test $T$,
$$
\pi\Eb_0[T(X)]+(1-\pi)\Eb_1[1-T(X)].
$$
\end{theorem}

The following lemmas are also useful to prove the lower bound.

\begin{lemma}\label{LBlemma1}
Let $b,c>0$.
Let $P$ and $Q$ be two probability measures having densities $p$ and $q$  with respect to some measure $\mu$.
It holds
$$
\int \min\{b p(x),c q(x) \} \dd \mu(x) \geq \frac{bc}{b+c}\left( \int \sqrt{p(x)q(x)}\dd \mu(x)\right)^2.
$$
\end{lemma}

The case $b=c=1$ can be found in \cite{TsybakovIntroNonparametricEstimation} (lemma 2.3).
We generalize the proof for any $b,c>0$.

\begin{proof}
The Cauchy-Schwarz inequality yields
\begin{align*}
bc\left( \int \sqrt{p(x)q(x)}\dd \mu(x)\right)^2&=\left( \int \sqrt{bp(x)\cdot cq(x)}\dd \mu(x)\right)^2\\
&=\left( \int \sqrt{\min\{bp(x),cq(x)\}}\sqrt{\max\{bp(x),cq(x)\}}\dd \mu(x)\right)^2\\
&\leq \int \min\{bp(x),cq(x)\}\dd \mu(x) \int \max\{bp(x),cq(x)\}\dd \mu(x).
\end{align*}
In order to finish, we use that $\max \{u, \,  v\} \leq u+v$ and get $\int \max\{bp(x),cq(x)\}\dd \mu(x) \leq b \int p(x) \dd \mu(x) + c\int q(x) \dd \mu(x) = b+c$.
\end{proof}

In the proof of the lower bound, Lemma \ref{LBlemma1} will be combined with the following result whose proof can be found in \cite{TsybakovIntroNonparametricEstimation}.

\begin{lemma}\label{LBlemma2}
Let $P$ and $Q$ be two probability measures having densities $p$ and $q$  with respect to some measure $\mu$.
It holds
$$
\left( \int \sqrt{p(x)q(x)}\dd \mu(x)\right)^2\geq \exp(-\KL(P,Q)).
$$
\end{lemma}

\subsection{Proof of Theorem \ref{Thrm Lower Bound Scenario 1}}\label{App Proof Lower bound Scenario 1}

Let $Q\in\Qc_\alpha^{CL}$ and let  $\hat{\eta}$ be a separable selector.
Since $\hat{\eta}_j$ depends only on $(Z^{i}_j)_{i=1,\ldots,n}$, it holds
$$
\Eb_{Q(P_\theta^{\otimes n})}\vert \hat{\eta}(Z)-\eta\vert = \sum_{j=1}^d \Eb_{\otimes_{i=1}^nQ^{i}_jP_{\theta_j}}\vert \hat{\eta}_j(Z^{1}_j,\ldots,Z^{n}_j )-\eta_j\vert.
$$
Following the proof of Theorem 2.2 in \cite{butucea2018variableSelectionHammingLoss}, we denote by $\Theta'$ the set of all $\theta$ in $\Theta_d^+(s,a)$ such that exactly $s$ components of $\theta$ are equal to $a$ and the remaining $d-s$ components are equal to $0$.
Since $\Theta'$ is a subset of $\Theta_d^+(s,a)$, it holds
\begin{align*}
\sup_{\theta\in\Theta_d^+(s,a)}\frac{1}{s}\Eb_{Q(P_\theta^{\otimes n})}\vert\hat{\eta}(Z)-\eta\vert & \geq \frac{1}{s\vert\Theta'\vert}\sum_{\theta\in\Theta'} \sum_{j=1}^d\Eb_{\otimes_{i=1}^nQ^{i}_jP_{\theta_j}}\vert \hat{\eta}_j(Z^{1}_j,\ldots,Z^{n}_j )-\eta_j\vert\\
&=\frac{1}{s\vert\Theta'\vert}\sum_{j=1}^d \left(\sum_{\theta\in\Theta':\theta_j=0}\Eb_{\otimes_{i=1}^nQ^{i}_jP_0}(\hat{\eta}_j)+\sum_{\theta\in\Theta':\theta_j=a}\Eb_{\otimes_{i=1}^nQ^{i}_jP_a}(1-\hat{\eta}_j)\right)\\
&=\frac{1}{s}\sum_{j=1}^d \left(\left(1-\frac{s}{d}\right)\Eb_{\otimes_{i=1}^nQ^{i}_jP_0}(\hat{\eta}_j)+\frac{s}{d}\Eb_{\otimes_{i=1}^nQ^{i}_jP_a}(1-\hat{\eta}_j)\right)\\
&\geq \frac{1}{s}\sum_{j=1}^d\inf_{T\in[0,1]}\left(\left(1-\frac{s}{d}\right)\Eb_{\otimes_{i=1}^nQ^{i}_jP_0}(T)+\frac{s}{d}\Eb_{\otimes_{i=1}^nQ^{i}_jP_a}(1-T)\right).
\end{align*}
Set
$$
L^*_j=\inf_{T\in[0,1]}\left(\left(1-\frac{s}{d}\right)\Eb_{\otimes_{i=1}^nQ^{i}_jP_0}(T)+\frac{s}{d}\Eb_{\otimes_{i=1}^nQ^{i}_jP_a}(1-T)\right).
$$
Since $Q^{i}_j$ provides $\alpha_j$-differential privacy, the channel probabilities $Q^{i}_j(\cdot \mid x)$  have densities $z\mapsto q^{i}_j(z\mid x)$ with respect to some measure $\mu^{i}_j$.
Therefore, $dQ^{i}_jP_0(z)=m^{i}_{j,0}(z)\dd \mu^{i}_j(z)$, and $dQ^{i}_jP_a(z)=m^{i}_{j,a}(z)\dd \mu^{i}_j(z)$, where $m^{i}_{j,b}(z)=\int_\Rb q^{i}_j(z\mid x)dP_b(x)$, $b\in\{0,a\}$.
Thus, for $b\in\{0,a\}$, it holds $$
d(\otimes_{i=1}^nQ^{i}_jP_b)(y_1,\ldots,y_n)=[\prod_{i=1}^n m^{i}_{j,b}(y_i)]\dd\mu_j(y_1,\ldots,y_n),
$$ where $\mu_j=\mu^1_j\otimes\cdots\otimes\mu^n_j$.
According to Theorem \ref{BayesianNeaymanPearson}, the infimum $L^*_j$  is thus attained for $T=T^*_j$ given by
$$T^*_j(Y_1,\ldots,Y_n)=I\left(\frac{s}{d}\prod_{i=1}^nm^{i}_{j,a}(Y_i)>\left(1-\frac{s}{d}\right)\prod_{i=1}^nm^{i}_{j,0}(Y_i)\right).$$
Set $A_j=\{(y_1,\ldots,y_n)\in\Rb^n : \frac{s}{d}\prod_{i=1}^nm^{i}_{j,a}(y_i)>\left(1-\frac{s}{d}\right)\prod_{i=1}^nm^{i}_{j,0}(y_i)\}$.
\begin{align*}
\sup_{\theta\in\Theta_d^+(s,a)}\frac{1}{s}\Eb_{Q(P_\theta^{\otimes n})}\vert\hat{\eta}(Z)-\eta\vert
& \geq \frac{1}{s}\sum_{j=1}^d\left[ \left(1-\frac{s}{d}\right)\int_{A_j} [\prod_{i=1}^nm^{i}_{j,0}(y_i)]\dd\mu_j(y_1,\ldots,y_n) \right.\\
& \hspace{1cm}\left.+\frac{s}{d}\int_{A_j^C} [\prod_{i=1}^nm^{i}_{j,a}(y_i)]\dd\mu_j(y_1,\ldots,y_n)\right]\\
&=\frac{1}{s}\sum_{j=1}^d\int_{\Rb^n} \min\left\{\left(1-\frac{s}{d}\right)\prod_{i=1}^nm^{i}_{j,0}(y_i),\frac{s}{d} \prod_{i=1}^nm^{i}_{j,a}(y_i)\right\}\dd\mu_j(y_1,\ldots,y_n)\\
&\geq \left(1-\frac{s}{d}\right)\cdot\frac{1}{d}\sum_{j=1}^d\left( \int_{\Rb^n}\sqrt{\left( \prod_{i=1}^nm^{i}_{j,0}(y_i)\right)\left( \prod_{i=1}^nm^{i}_{j,a}(y_i)\right)}\dd\mu_j(y_1,\ldots,y_n) \right)^2\\
&\geq \left(1-\frac{s}{d}\right)\cdot\frac{1}{d}\sum_{j=1}^d\exp\left(-\KL\left(\otimes_{i=1}^nQ^{i}_jP_0 ,\otimes_{i=1}^nQ^{i}_jP_a \right) \right),
\end{align*}
where the two last inequalities follow from lemma \ref{LBlemma1} and lemma \ref{LBlemma2}.
Using the tensorization property of the Kullback-Leibler divergence we obtain
\begin{align*}
\sup_{\theta\in\Theta_d^+(s,a)}\frac{1}{s}\Eb_{Q(P_\theta^{\otimes n})}\vert\hat{\eta}(Z)-\eta\vert &\geq \left(1-\frac{s}{d}\right)\cdot\frac{1}{d}\sum_{j=1}^d\exp\left(-\sum_{i=1}^n\KL\left(Q^{i}_jP_0,Q^{i}_jP_a \right) \right)\\
&\geq \left(1-\frac{s}{d}\right)\cdot\frac{1}{d}\sum_{j=1}^d\exp\left(-4n (e^{\alpha/d}-1)^2\TV(P_0,P_a)^2\right)\\
&=\left(1-\frac{s}{d}\right)\exp\left(-4n (e^{\alpha/d}-1)^2\TV(P_0,P_a)^2\right),
\end{align*}
where the second inequality is a direct consequence of the strong data processing inequality in Theorem 1 of \cite{DuchiJordanWainwright2018MinimaxOptimalProcedure} showing the contraction property of privacy: $$\KL(QP_0, QP_a) + \KL(QP_a, QP_0) \leq (4 \wedge e^{2\alpha/d}) (e^{\alpha/d} - 1)^2 TV(P_0,P_a)^2$$ if $Q$ is $\alpha/d-$DP.
Since this result holds for all $Q\in \Qc_\alpha^{CL}$ and all separable selector $\hat{\eta}$, we obtain
$$
\inf_{Q\in \Qc_{\alpha}^{CL}}\inf_{\hat{\eta}\in\Tc}\sup_{\theta\in\Theta_d^+(s,a)}\frac{1}{s}\Eb_{Q(P_\theta^{\otimes n})}\vert\hat{\eta}-\eta\vert\geq \left(1-\frac{s}{d}\right)\exp\left(-4n (e^{\alpha/d}-1)^2\TV(P_0,P_a)^2\right).
$$
Note that the $TV$ distance is invariant to a scale parameter, thus $TV(P_0,P_a)$ can be calculated as the $TV$ distance between the distribution of $\xi_1$ and the same one shifted by $a/\sigma$.
The inequality $TV(P_0,P_a)\leq 1$, Pinsker's inequality and Inequality \eqref{eq:upper_bound_KL} of Lemma \ref{lem:Bound_KL} below, give
$$
TV(P_0,P_a)\leq \sqrt{\frac{\KL(P_0,P_a)}{2}}\leq \frac{a\sqrt{c_+}}{2\sigma},
$$
which implies the statement of Theorem \ref{Thrm Lower Bound Scenario 1}.

\begin{lemma}
\label{lem:Bound_KL}Consider that the measure $P^{\xi_1}$ of the noise coordinates is $c-$strongly log-concave
on $\mathbb{R}$, with density $p=\exp(-\phi)$, the convex function
$\phi$ thus being $c-$strongly convex for some constant $c>0$.
Recall that the measure $P_{0}$ is the distribution of the scaled noise coordinate $\sigma \xi_1$ and that $P_{a}$ is the image of $P_{0}$ by the translation
of $a$. It holds
\begin{equation}
{\rm KL}(P_{0},P_{a})\geq\frac{ca^{2}}{2\sigma^2}.\label{eq:lower_bound_KL}
\end{equation}
If, on the other hand, we assume that the measure $P^{\xi_1}$ has a density $p=\exp(-\phi)$ converging to zero at infinity, with $\phi$ being two times continuously differentiable and satisfying Inequality \eqref{eq:upper_bound_hessian_intro} for a constant $c_+>0$, that gives a uniform upper bound of the curvature of $\phi$ by the constant $c_+$, then it holds
\begin{equation}
{\rm KL}(P_{0},P_{a})\leq\frac{c_{+}a^{2}}{2\sigma^2}.\label{eq:upper_bound_KL}
\end{equation}
If $P^{\xi_1}$ is $c-$strongly log-concave and its potential achieves \eqref{eq:upper_bound_hessian_intro}
for a positive constant $c_+$, then Inequality \eqref{eq:upper_bound_KL} holds true. \end{lemma}

Note that Lemma \ref{lem:Bound_KL} is tight in the sense that if
$P_{0}$ is Gaussian with variance $\sigma^{2}$, then ${\rm KL}(P_{0},P_{a})=a^{2}/(2\sigma^{2})$
and we have equality in both bounds (\ref{eq:lower_bound_KL}) and (\ref{eq:upper_bound_KL}), with $c=c_{+}=1$.
\begin{proof}
Consider first the case of a $c-$strongly log-concave density $p$.By standard approximation arguments (see for instance \cite[Proposition 5.5]{SauWell14}), we can assume without loss of generality that the potential $\phi$
is two times continuously differentiable. Consequently, $c-$strong
convexity of $\phi$ is equivalent to $\phi^{\prime\prime}(x)\geq c$
for all $x\in\mathbb{R}$. Using Taylor expansion, this gives that
for any $x\in\mathbb{R}$,
\[
\phi(x-a)\geq\phi(x)-a\phi^{\prime}(x)+c\frac{a^{2}}{2}.
\]
This gives,
\begin{align*}
{\rm KL}(P_{0},P_{a}) & =\frac{1}{\sigma^2}\int_{\mathbb{R}}[\phi(x-a)-\phi(x)]\exp(-\phi(x))dx\\
 & \geq-\frac{a}{\sigma^2}\int_{\mathbb{R}}\phi^{\prime}(x)\exp(-\phi(x))dx+c_{-}\frac{a^{2}}{2\sigma^2}\int_{\mathbb{R}}\exp(-\phi(x))dx\\
 & =c_{-}\frac{a^{2}}{2\sigma^2}.
\end{align*}
Hence, (\ref{eq:lower_bound_KL}) is proved. In order to prove (\ref{eq:upper_bound_KL}), just remark that  the regularity of $\phi$ and the upper bound on its curvature yield $\sup_{x\in \mathbb{R}}\phi^{\prime\prime}(x)\leq c_{+}$. Hence, by Taylor expansion again,
\[
\phi(x-a)\leq\phi(x)-a\phi^{\prime}(x)+c_{+}\frac{a^{2}}{2}.
\]
Analogous computations now give ${\rm KL}(P_{0},P_{a})\leq c_{+}a^{2}/(2\sigma^2)$, which is \eqref{eq:upper_bound_KL}.

Finally, if $p$ is strongly log-concave, then it tends to zero at infinity. By convolution by Gaussians, we can also assume without loss of generality that $\phi$ is two times continuously differentiable (\cite[Proposition 5.5]{SauWell14}). Hence, if in addition $p$ satisfies \eqref{eq:upper_bound_hessian_intro}, then it achieves the conditions that lead to \eqref{eq:upper_bound_KL}. This concludes the proof of Lemma \ref{lem:Bound_KL}.
\end{proof}

\subsection{Some auxiliary results for the upper bounds}\label{App Auxiliary results upper bounds scenario 1}

First recall that since the vector $\xi=(\xi_1,...,\xi_d)$ is $c$-strongly log-concave - as it has independent $c$-strongly log-concave coordinates -, then it achieves the following sub-Gaussian concentration inequality for Lipschitz functions (see Proposition 2.18 in \cite{Ledoux01}): for any $L$-Lipschitz function $F:\mathbb{R}^d \rightarrow \mathbb{R}$, and any $r\geq 0$,
\begin{equation}
    \mathbb{P}(F(\xi)-\mathbb{E}[F(\xi)]\geq r)\leq \exp(-cr^2/(2L^2)).\label{ineq_Gauss_Lip}
\end{equation}
Furthermore, the celebrated Cafarelli's contraction theorem \cite{Caff00} state that the Brenier transport map pushing forward a Gaussian distribtion to a strongly log-concave measure with the same Gaussian factor is a contraction. As a result, one can derive the following Mill's type bound for the deviations the coordinate $\xi_1$ (\cite{Cordero02}, Proposition 2): for any $r>0$,
\begin{equation}
    \mathbb{P}(|\xi_1|\geq r)\leq 2(1-\Phi(\sqrt{c}r))\leq \sqrt{\frac{2}{\pi}}\frac{\exp(-cr^2/2)}{\sqrt{c}r},\label{ineq_Mills}
\end{equation}
where $\Phi$ is the standard Gaussian cumulative distribution function. Another useful fact is that, as $\xi_1$ is log-concave, it is unimodal (see \cite{SauWell14}) and as $\xi_1$ is also symmetric, the maximum of its density $p(x)=\exp(-\phi_0(x))\exp(-cx^2/2)$ is attained at its median $0$. Note that as $\phi_0$ is convex and symmetric, the maximum of the function $\exp(-\phi_0(x))$ is also attained at $0$. In addition, by a result of Bobkov and Ledoux \cite{BobLed19} (see also \cite{SauWell14} Proposition 5.2), as $0$ is the median of $\xi_1$, it holds $p(0)=\exp(-\phi_0(0))\leq 1/(\sqrt{2}\sigma_{\xi_1})\leq 1/\sqrt{2}$. Putting things together we obtain that for any $x\in \mathbb{R}$,
\begin{equation}
    p(x)\leq \frac{1}{\sqrt{2}}\exp(-cx^2/2).\label{upper_bound_p}
\end{equation}

\begin{lemma}\label{Lemme Troncation Sous Gaussienne}
For all $a\geq 0$ and all $\tau>0$, by Hoeffding inequality we have
$$
\Pb\left( \frac{1}{n}\sum_{i=1}^n\left(\sgn[a+\sigma\xi^{i}_j]- \Eb\left[\sgn[a+\sigma\xi^{i}_j] \right] \right) \geq \tau \right)
\leq \exp(- n\frac{\tau^2}{2}).
$$ 
\end{lemma}


We now recall Bernstein's inequality (cf. \cite{boucheron2013concentration} Corollary 2.11).
\begin{theorem}
Let $Y_1,\ldots,Y_n$ be independent real valued random variables. Assume that there exist some positive numbers $v$ and $u$ such that
\begin{equation}
\sum_{i=1}^n\mathbb{E}[Y_i^2]\leqslant v,
\end{equation}
and for all integers $m\geqslant 3$
\begin{equation}
\sum_{i=1}^n \mathbb{E}[\vert Y_i\vert^m]\leqslant \frac{m!}{2}vu^{m-2}.
\end{equation}
Let $S=\sum_{i=1}^n(Y_i-\mathbb{E}[Y_i])$, then for every positive $t$
\begin{equation}
\label{BernsteinsInequality}
\mathbb{P}\left(S\geqslant t\right)\leqslant  \exp\left(-\frac{t^2}{2(v+ut)}\right).
\end{equation}
\end{theorem}
Note that if $v\leq ut$ then \eqref{BernsteinsInequality} yields $\Pb(S\geq t)\leq \exp(-t/(4u))$.
If $ut\leq v$ then \eqref{BernsteinsInequality} yields $\Pb(S\geq t)\leq \exp(-t^2/4v)$.

We will apply this to get concentration bounds for the average of i.i.d. Laplace distributed random variables that check the assumptions of the Theorem.
\subsection{Proof of Proposition \ref{Prop Upper Bound on the risk for large a Scenario 1}}\label{App proof Upper Bound Large a scenario 1}

It holds
\begin{align*}
\vert \hat{\eta}^+-\eta\vert&=\sum_{j : \eta_j=0}\hat{\eta}^+_j+\sum_{j : \eta_j=1}(1-\hat{\eta}^+_j)\\
&=\sum_{j : \eta_j=0}I\left(\frac{1}{n}\sum_{i=1}^n \sgn[\sigma\xi^{i}_j]+\frac{2d}{n\alpha}\sum_{i=1}^nW^{i}_j\geq \tau\right)\\
&\hspace{1.5cm}+\sum_{j : \eta_j=1}I\left(\frac{1}{n}\sum_{i=1}^n \sgn[\theta_j+\sigma\xi^{i}_j] +\frac{2 d}{n\alpha}\sum_{i=1}^nW^{i}_j < \tau\right).
\end{align*}
Thus,
\begin{align*}
\Eb\left[\frac{1}{s}\vert\hat{\eta}^+-\eta\vert \right]
&=\frac{1}{s}\sum_{j : \eta_j=0}\underbrace{\Pb\left(\frac{1}{n}\sum_{i=1}^n \sgn[\sigma\xi^{i}_j]+\frac{2d}{n\alpha}\sum_{i=1}^nW^{i}_j\geq \tau\right)}_{=T_{1,j}}\\
&\hspace{1.5cm}+\frac{1}{s}\sum_{j : \eta_j=1}\underbrace{\Pb\left(\frac{1}{n}\sum_{i=1}^n \sgn[\theta_j+\sigma\xi^{i}_j] +\frac{2d}{n\alpha}\sum_{i=1}^nW^{i}_j < \tau\right)}_{=T_{2,j}}.
\end{align*}
We first study $T_{1,j}$.
It holds
$$
T_{1,j} \leq \Pb\left(\frac{1}{n}\sum_{i=1}^n \sgn[\sigma \xi^{i}_j]\geq \frac{\tau}{2} \right)+\Pb\left(\sum_{i=1}^nW^{i}_j\geq \frac{\tau n \alpha}{4d} \right).
$$
Note that $\Eb\left[\sgn[\sigma\xi_j^{i}] \right]=0$ by symmetry of $\xi_j^i$.
Using Lemma \ref{Lemme Troncation Sous Gaussienne} to bound from above the first term and Bernstein's inequality \eqref{BernsteinsInequality} with $v=2n$ and $u=1$ to bound from above the second term, we obtain if $\tau\alpha/(8d)<1$
$$
T_{1,j}\leq \exp\left(-\frac{n\tau^2}{2^3}\right)+\exp\left(-\frac{\tau^2n\alpha^2}{2^7 d^2}  \right).
$$
Since $x\mapsto \sgn[x]$ is a non-decreasing function and since $\theta_j\geq a$ for all $j$ such that $\eta_j=1$, it holds
\begin{align*}
T_{2,j}&\leq\Pb\left(\frac{1}{n}\sum_{i=1}^n \sgn[a+\sigma\xi^{i}_j] +\frac{2d}{n\alpha}\sum_{i=1}^nW^{i}_j < \tau\right)\\
&=\Pb\left(\frac{1}{n}\sum_{i=1}^n\left( \sgn[a+\sigma\xi^{i}_j]- \Eb\left[ \sgn[a+\sigma\xi^{i}_j] \right] \right)+\Eb\left[\sgn[a+\sigma\xi^{1}_j] \right]+\frac{2d}{n\alpha}\sum_{i=1}^nW^{i}_j < \tau\right)\\
&=\Pb\left(-\frac{1}{n}\sum_{i=1}^n\left(\sgn[a+\sigma\xi^{i}_j]- \Eb\left[\sgn[a+\sigma\xi^{i}_j] \right] \right)-\frac{2d}{n\alpha}\sum_{i=1}^nW^{i}_j > \Eb\left[\sgn[a+\sigma\xi^{1}_j]\right]-\tau \right).
\end{align*}
As $\xi_1$ is symmetric and absolutely continuous, we have 
\begin{align*}
\Eb\left[\sgn[a+\sigma \xi_1]\right]
&= \Pb \left( \xi_1 \geq -\frac{a}{\sigma}\right)-\Pb \left( \xi_1 < -\frac{a}{\sigma}\right)\\
&= 1-2 \Pb \left( \xi_1 > \frac{a}{\sigma}\right).
\end{align*}
Using (\ref{ineq_Mills}) we further get 
\begin{align*}
\Eb\left[\sgn[a+\sigma \xi_1]\right]
& \geq 2 \Phi \left( \sqrt{c}\frac{a}{\sigma}\right)-1
 \geq 2 \Phi(2 \sqrt{c})-1=:C_1,
\end{align*}
for $a/\sigma \geq 2$, with $\Phi$ the c.d.f. of the standard Gaussian distribution.

Thus, if $a\geq 2\sigma $, it holds
\begin{align*}
T_{2,j}&\leq \Pb\left(-\frac{1}{n}\sum_{i=1}^n\left(\sgn[a+\sigma\xi^{i}_j]- \Eb\left[\sgn[a+\sigma\xi^{i}_j] \right] \right)-\frac{2d}{n\alpha}\sum_{i=1}^nW^{i}_j > C_1-\tau\right)\\
&\leq \Pb\left(-\frac{1}{n}\sum_{i=1}^n\left(\sgn[a+\sigma\xi^{i}_j]- \Eb\left[\sgn[a+\sigma\xi^{i}_j] \right] \right) > \frac{C_1-\tau}{2}\right)\\
&\hspace{5cm} + \Pb\left(\sum_{i=1}^n\left(-W^{i}_j \right)> \frac{n\alpha(C_1 -\tau)}{4d}\right).
\end{align*}
We can now bound from above the first term using lemma \ref{Lemme Troncation Sous Gaussienne} and the second term using Bernstein's inequality.
This gives, if $C_1\geq \tau$ and $\alpha(C_1-\tau)/(8d)\leq 1$
$$
T_{2,j}\leq \exp\left(-\frac{n(C_1-\tau)^2}{2^3 }\right)+\exp\left(-\frac{(C_1-\tau)^2n\alpha^2}{2^7 d^2}  \right).
$$
This ends the proof of \eqref{Eq Upper bound large a scenario 1 theta+}.
We now prove \eqref{Eq Upper bound large a scenario 1 theta}.
If $\theta\in\Theta_d(s,a)$, we use the estimator $\hat{\eta}$ instead of $\hat{\eta}^+$ and it holds
\begin{align*}
\Eb\left[\frac{1}{s}\vert\hat{\eta}-\eta\vert \right]
&=\frac{1}{s}\sum_{j : \eta_j=0}\underbrace{\Pb\left(\left\vert\frac{1}{n}\sum_{i=1}^n \sgn[\sigma\xi^{i}_j]
+\frac{2d}{n\alpha}\sum_{i=1}^nW^{i}_j\right\vert\geq \tau\right)}_{=\tilde{T}_{1,j}}\\
&\hspace{1.5cm}+\frac{1}{s}\sum_{j : \eta_j=1}\underbrace{\Pb\left(\left\vert\frac{1}{n}\sum_{i=1}^n \sgn[\theta_j+\sigma\xi^{i}_j] +\frac{2d}{n\alpha}\sum_{i=1}^nW^{i}_j\right\vert < \tau\right)}_{=\tilde{T}_{2,j}}.
\end{align*}
We first study $\tilde{T}_{1,j}$.
It holds
\begin{align*}
 \tilde{T}_{1,j}&\leq \Pb\left( \left\vert\frac{1}{n}\sum_{i=1}^n \sgn[\sigma\xi^{i}_j] \right\vert+\frac{2d}{n\alpha}\left\vert\sum_{i=1}^nW^{i}_j\right\vert\geq \tau \right)\\
 &\leq \Pb\left(\left\vert\frac{1}{n}\sum_{i=1}^n \sgn[\sigma \xi^{i}_j] \right\vert\geq \frac{\tau}{2} \right)+\Pb\left(\left\vert\sum_{i=1}^nW^{i}_j\right\vert\geq \frac{\tau n \alpha}{4d} \right)\\
 &\leq 2\left[\exp\left(-\frac{n\tau^2}{2^3 }\right)+\exp\left(-\frac{\tau^2n\alpha^2}{2^7 d^2}  \right)\right],
\end{align*}
if $\tau\alpha/(8d)<1$, where we have used the two-sided versions of the concentration inequalities we used to prove \eqref{Eq Upper bound large a scenario 1 theta+}.
We now study $\tilde{T}_{2,j}$.
For all $j$ such that $\eta_j=1$, it holds
\begin{align*}
\tilde{T}_{2,j}&=\Pb\left(\left\vert\frac{1}{n}\sum_{i=1}^n\left( \sgn[\theta_j+\sigma\xi^{i}_j]- \Eb\left[\sgn[\theta_j+\sigma\xi^{i}_j] \right] \right)+\Eb\left[\sgn[\theta_j+\sigma\xi^{1}_j] \right]+\frac{2d}{n\alpha}\sum_{i=1}^nW^{i}_j\right\vert < \tau\right)\\
&\leq \Pb\left(\left\vert \Eb\left[\sgn[\theta_j+\sigma\xi^{1}_j] \right] \right\vert -\left\vert\frac{1}{n}\sum_{i=1}^n\left(\sgn[\theta_j+\sigma\xi^{i}_j]- \Eb\left[\sgn[\theta_j+\sigma\xi^{i}_j] \right] \right)+\frac{2d}{n\alpha}\sum_{i=1}^nW^{i}_j\right\vert < \tau\right)\\
&= \Pb\left(\left\vert\frac{1}{n}\sum_{i=1}^n\left( \sgn[\theta_j+\sigma\xi^{i}_j]- \Eb\left[\sgn[\theta_j+\sigma\xi^{i}_j] \right] \right)+\frac{2d}{n\alpha}\sum_{i=1}^nW^{i}_j\right\vert > \left\vert \Eb\left[\sgn[\theta_j+\sigma\xi^{1}_j] \right] \right\vert-\tau \right)
\end{align*}
Now, observe that
$$
\left\vert \Eb\left[ \sgn[\theta_j+\sigma\xi^{1}_j] \right] \right\vert 
\geq  \Eb\left[ \sgn[\theta_j+\sigma\xi^{1}_j] \right] 
\geq  \Eb\left[ \sgn[a+\sigma\xi^{1}_j] \right],
$$
if $\theta_j\geq a$ since $x\mapsto \sgn[x]$ is non-decreasing, and if $\theta_j\leq -a$ we have
\begin{align*}
&\left\vert \Eb\left[ \sgn[\theta_j+\sigma\xi^{1}_j] \right] \right\vert
\geq  -\Eb\left[\sgn[\theta_j+\sigma\xi^{1}_j] \right]
\geq  -\Eb\left[\sgn[-a+\sigma\xi^{1}_j] \right]  \\
& = -\Eb\left[\sgn[-a -\sigma\xi^{1}_j] \right]
=\Eb\left[ \sgn[a+\sigma\xi^{1}_j] \right],
\end{align*}
where we have used that $x\mapsto \sgn[x]$ is a non-decreasing and odd function and that $-\xi^{1}_j$ and $\xi^{1}_j$ have the same distribution.
Moreover, we have seen in the proof of \eqref{Eq Upper bound large a scenario 1 theta+} that it holds
$$
\Eb\left[ \sgn[a+\sigma \xi] \right]\geq 2\Phi\left(2 \sqrt{c}\right)-1=: C_1 ,
$$
where $\Phi$ denotes the cumulative distribution function of the Gaussian distribution.
Thus, if $a\geq 2\sigma $, it holds $\Eb\left[\sgn[a+\sigma \xi^{1}_j]\right]\geq C_1$ for all $j$ such that $\eta_j=1$, and
\begin{align*}
\tilde{T}_{2,j}&\leq \Pb\left(\left\vert\frac{1}{n}\sum_{i=1}^n\left( \sgn[\theta_j+\sigma\xi^{i}_j]
- \Eb\left[ \sgn[\theta_j+\sigma\xi^{i}_j] \right] \right)+\frac{2d}{n\alpha}\sum_{i=1}^nW^{i}_j\right\vert > C_1-\tau \right)\\
&\leq \Pb\left(\left\vert\frac{1}{n}\sum_{i=1}^n\left( \sgn[\theta_j+\sigma\xi^{i}_j] - \Eb\left[\sgn[\theta_j+\sigma\xi^{i}_j] \right] \right)\right\vert+\left\vert\frac{2d}{n\alpha}\sum_{i=1}^nW^{i}_j\right\vert > C_1-\tau \right)\\
&\leq \Pb\left(\left\vert\frac{1}{n}\sum_{i=1}^n\left(\sgn[\theta_j+\sigma\xi^{i}_j]- \Eb\left[\sgn[\theta_j+\sigma\xi^{i}_j] \right] \right)\right\vert > \frac{C_1-\tau}{2}\right)\\
&\hspace{5cm} + \Pb\left(\left\vert\sum_{i=1}^n W^{i}_j\right\vert> \frac{n\alpha(C_1-\tau)}{4d}\right).
\end{align*}
Using the two-sided version of the concentration inequalities that we used to bound $T_{2,j}$ in the proof of \eqref{Eq Upper bound large a scenario 1 theta+}, we obtain if $C_1 > \tau$ and $\alpha(C_1 -\tau)/(8d)\leq 1$
$$
\tilde{T}_{2,j}\leq 2\left[ \exp\left(-\frac{n(C_1-\tau)^2}{2^3} \right)+\exp\left(-\frac{(C_1-\tau)^2n\alpha^2}{2^7 d^2}  \right) \right].
$$
This ends the proof of \eqref{Eq Upper bound large a scenario 1 theta}.

\subsection{Proof of Corollary \ref{Corollary Sufficient Conditions for AFR large a scenario 1}}\label{App Corollary Upper Bound Large a scenario 1}

Let us prove \eqref{Eq AFR large a scenario 1}.
Note that if the assumptions of Corollary \ref{Corollary Sufficient Conditions for AFR large a scenario 1} are satisfied, and if $\tau=C_1/2$ then the assumptions of Proposition \ref{Prop Upper Bound on the risk for large a Scenario 1} are also satisfied and for all $a\geq 2\sigma$ we have
\begin{align*}
\sup_{\theta\in\Theta}\Eb\left[\frac{1}{s}\vert\hat{\eta}^+-\eta\vert \right]&\leq 2\cdot\frac{d}{s}\left[\exp\left(-\frac{C_1^2n}{2^5}\right)+\exp\left(-\frac{C_1^2n\alpha^2}{2^9 d^2}  \right)\right]\\
&=2\left\{\exp\left(\log\left(\frac{d}{s} \right)-\frac{C_1^2n}{2^5}\right)+\exp\left(\log\left(\frac{d}{s} \right)-\frac{C_1^2n\alpha^2}{2^9 d^2}  \right)\right\}\\
&=2\exp\left(-\frac{n\alpha^2}{d^2}\left[\frac{C_1^2d^2}{2^5\alpha^2}-\frac{\log(d/s)}{n\alpha^2/d^2}\right]\right)\\
&+2\exp\left(-\frac{n\alpha^2}{d^2}\left[\frac{C_1^2}{2^9}-\frac{\log(d/s)}{n\alpha^2/d^2}\right]\right).
\end{align*}
The two terms appearing in the last inequality both tend to $0$ as $d\rightarrow +\infty$ under the assumptions of Corollary \ref{Corollary Sufficient Conditions for AFR large a scenario 1}, which gives \eqref{Eq AFR large a scenario 1}.
The proof of \eqref{Eq ER large a scenario 1} is similar.

\subsection{Proof of Proposition \ref{Prop Upper Bound for small a scenario 1}}\label{App Proof Upper Bound Small a scenario 1}

The beginning of the proof is similar to the proof of Proposition \ref{Prop Upper Bound on the risk for large a Scenario 1}.
It holds
\begin{align*}
\Eb\left[\frac{1}{s}\vert\hat{\eta}^+-\eta\vert \right]&=\frac{1}{s}\sum_{j : \eta_j=0}\underbrace{\Pb\left(\frac{1}{n}\sum_{i=1}^n \sgn[\sigma\xi^{i}_j] +\frac{2d}{n\alpha}\sum_{i=1}^nW^{i}_j\geq \tau\right)}_{=T_{1,j}}\\
&\hspace{1.5cm}+\frac{1}{s}\sum_{j : \eta_j=1}\underbrace{\Pb\left(\frac{1}{n}\sum_{i=1}^n \sgn[\theta_j+\sigma\xi^{i}_j] +\frac{2d}{n\alpha}\sum_{i=1}^nW^{i}_j < \tau\right)}_{=T_{2,j}},
\end{align*}
and we have
$$
T_{1,j}\leq \exp\left(-\frac{n\tau^2}{2^3} \right)+\exp\left(-\frac{\tau^2n\alpha^2}{2^7 d^2}  \right)
$$
if $\tau\alpha/(8d)<1$, and
$$
T_{2,j}\leq\Pb\left(-\frac{1}{n}\sum_{i=1}^n\left(\sgn[a+\sigma\xi^{i}_j]- \Eb\left[\sgn[a+\sigma\xi^{i}_j] \right] \right)-\frac{2d}{n\alpha}\sum_{i=1}^nW^{i}_j > \Eb\left[\sgn[a+\sigma\xi^{1}_j]\right]-\tau \right).
$$
Now, we bound from below $\Eb\left[\sgn[a+\sigma \xi]\right]$ in a different way than in the proof of Proposition \ref{Prop Upper Bound on the risk for large a Scenario 1} by the tighter bound:
\begin{align*}
\Eb\left[ \sgn[a+\sigma \xi_1]\right]&= 2 \Pb(0<\xi_1 \leq \frac a\sigma) \geq 2 \frac a\sigma p(\frac a \sigma)\geq 2 \frac a\sigma p(2),
\end{align*}
for $a/\sigma <2$, as the pdf $p$ of $\xi_1$ is $c-$strongly log-concave and symmetric and thus uni-modal at 0 and decreasing on $(0,\infty)$.

Thus, 
\begin{align*}
T_{2,j}&\leq \Pb\left(-\frac{1}{n}\sum_{i=1}^n\left(\sgn[a+\sigma\xi^{i}_j]- \Eb\left[\sgn[a+\sigma\xi^{i}_j] \right] \right)-\frac{2d}{n\alpha}\sum_{i=1}^nW^{i}_j > 2\frac{a}{\sigma}p(2)-\tau\right)\\
&\leq \Pb\left(-\frac{1}{n}\sum_{i=1}^n\left(\sgn[a+\sigma\xi^{i}_j]- \Eb\left[\sgn[a+\sigma\xi^{i}_j] \right] \right) > \frac a \sigma p(2)-\frac \tau 2\right)\\
&\hspace{5cm} + \Pb\left(\sum_{i=1}^n\left(-W^{i}_j \right)> \frac{n\alpha(a/\sigma p(2)-\tau/2)}{2d}\right).
\end{align*}
We can now bound from above the first term using lemma \ref{Lemme Troncation Sous Gaussienne} and the second term using Bernstein's inequality.
This gives, if $\tau < 2 a/\sigma p(2) $ and $\alpha(a/\sigma p(2)-\tau/2)/(4d)\leq 1$
$$
T_{2,j}\leq \exp\left(-\frac{n(a/\sigma p(2)-\tau/2)^2}{2^3}\right)
+\exp\left(-\frac{(a/\sigma p(2)-\tau/2)^2n\alpha^2}{2^5 d^2}  \right).
$$

\subsection{Proof of Corollary \ref{Corollary Sufficient Conditions for AFR small a}}\label{App Corollary upper bound small a scenario 1}

Let prove \eqref{Eq AFR small a scenario 1}.
The chosen value of $\tau = a/\sigma \cdot p(2)$  satisfies the assumptions of Proposition \ref{Prop Upper Bound for small a scenario 1} for $d$ large enough and yield
\begin{align*}
\sup_{\theta\in\Theta_d^+(s,a)}\Eb\left[\frac{1}{s}\vert\hat{\eta}^+-\eta\vert \right]
&\leq \frac{d}{s}\left[\exp\left(-\frac{na^2}{2^3 \sigma^2} p^2(2)\right)+\exp\left(-\frac{n\alpha^2 a^2}{2^{7} \sigma^2 d^2} p^2(2) \right)\right]\\
&=\exp\left(\log\left(\frac{d}{s} \right)-\frac{na^2}{2^3\sigma^2} p^2(2)\right)+\exp\left(\log\left(\frac{d}{s} \right)-\frac{n\alpha^2 a^2}{2^{7}\sigma^2 d^2} p^2(2) \right)\\
&\leq 2\exp\left(\log\left(\frac{d}{s} \right)-\frac{n\alpha^2 a^2}{2^{7}  \sigma^2 d^2}  \right).
\end{align*}
Conclude using that $a \gg \sigma d/\sqrt{n \alpha^2} \sqrt{\log(d/s)}$. 
The proof of \eqref{Eq ER small a scenario 1} is similar.

\section{Proofs of Section \ref{Section Scenario 2}}

\subsection{Proof of Proposition \ref{Prop alpha DP scenario 2}}\label{App Proof alpha DP scenario 2}

Note that it is sufficient to prove that $\tilde{Z}^{i}$ is an $\alpha$-LDP view of $X^{i}$.
Indeed, if $\tilde{Z}^{i}$ is an $\alpha$-LDP view of $X^{i}$ then it holds for all $z\in \Zc$ and  $x,x'\in\Rb^d$ (we omit the superscript $i$)
\begin{align*}
\frac{\Pb\left( Z=z\mid X=x\right)}{\Pb\left( Z=z\mid X=x'\right)}&=\frac{\sum_{\tilde{z}\in\{-B,B\}^d}\Pb\left( Z=z\mid \tilde{Z}=\tilde{z}, X=x \right)\Pb\left(\tilde{Z}=\tilde{z} \mid X=x  \right)}{\sum_{\tilde{z}\in\{-B,B\}^d}\Pb\left( Z=z\mid \tilde{Z}=\tilde{z}, X=x' \right)\Pb\left(\tilde{Z}=\tilde{z} \mid X=x ' \right)}\\
&=\frac{\sum_{\tilde{z}\in\{-B,B\}^d}\Pb\left( Z=z\mid \tilde{Z}=\tilde{z} \right)\Pb\left(\tilde{Z}=\tilde{z} \mid X=x  \right)}{\sum_{\tilde{z}\in\{-B,B\}^d}\Pb\left( Z=z\mid \tilde{Z}=\tilde{z} \right)\Pb\left(\tilde{Z}=\tilde{z} \mid X=x ' \right)}\\
&\leq e^\alpha,
\end{align*}
where we have used that $Z$ is independent from $X$ conditionally to $\tilde{Z}$ and the fact that $\Pb\left(\tilde{Z}=\tilde{z} \mid X=x  \right)\leq e^\alpha \Pb\left(\tilde{Z}=\tilde{z} \mid X=x ' \right)$ for all $\tilde{z}\in\{-B,B\}^d$ if $\tilde{Z}$ is an $\alpha$-LDP view of $X$.
So, let's prove that $\tilde{Z}^{i}$ is an $\alpha$-LDP view of $X^{i}$.
In what follows, we omit once again the superscript $i$.
We have to prove that for all $\tilde{z}\in\{-B,B\}^d$ and all $x,x'\in\Rb^d$ it holds
$$
\frac{\Pb\left(\tilde{Z}=\tilde{z} \mid X=x \right)}{\Pb\left(\tilde{Z}=\tilde{z} \mid X=x' \right)}\leq e^\alpha.
$$
Let $\tilde{z}\in\{-B,B\}^d$ and  $x\in\Rb^d$.
It holds
\begin{align*}
\Pb\left(\tilde{Z}=\tilde{z} \mid X=x \right)&=\sum_{\tilde{x}\in\{-1,1\}^d}\Pb\left(\tilde{Z}=\tilde{z} \mid X=x, \tilde{X}=\tilde{x} \right)\cdot \Pb\left( \tilde{X}=\tilde{x}\mid X=x \right)\\
&=\sum_{\tilde{x}\in\{-1,1\}^d}\Pb\left(\tilde{Z}=\tilde{z} \mid \tilde{X}=\tilde{x} \right)\cdot \Pb\left( \tilde{X}=\tilde{x}\mid X=x \right),\\
\end{align*}
and since $Y$ and $\tilde{X}$ are independent we have
\begin{align*}
\Pb\left(\tilde{Z}=\tilde{z} \mid \tilde{X}=\tilde{x} \right)&=\Pb\left(\tilde{Z}=\tilde{z} \mid \tilde{X}=\tilde{x},Y=1 \right)\cdot \Pb\left(Y=1 \right)\\
&\hspace{1cm}+\Pb\left(\tilde{Z}=\tilde{z} \mid \tilde{X}=\tilde{x}, Y=0 \right)\cdot \Pb\left(Y=0 \right)\\
&=\pi_\alpha\Pb\left(\tilde{Z}=\tilde{z} \mid \tilde{X}=\tilde{x},Y=1 \right)+(1-\pi_\alpha)\Pb\left(\tilde{Z}=\tilde{z} \mid \tilde{X}=\tilde{x}, Y=0 \right).
\end{align*}
Moreover, since for $\tilde{x}\in\{-1,1\}^d$
\begin{multline*}
\Card\left(\left\{ \tilde{z}\in\left\{-B,B\right\}^d \, \mid \, \langle \tilde{z}, \tilde{x} \rangle> 0 \text{ or } (\langle \tilde{z}, \tilde{x}\rangle=0 \text{ and }\tilde{z}_1=B\tilde{x}_1)\right\} \right) \\
=  \Card\left(\left\{ \tilde{z}\in\left\{-B,B\right\}^d \, \mid \, \langle \tilde{z}, \tilde{x} \rangle< 0 \text{ or } (\langle \tilde{z}, \tilde{x}\rangle=0 \text{ and }\tilde{z}_1=-B\tilde{x}_1)\right\} \right)  = 2^{d-1},
\end{multline*}
it holds
$$
\Pb\left(\tilde{Z}=\tilde{z} \mid \tilde{X}=\tilde{x},Y=1 \right)=
\left\lbrace \begin{array}{ll}
0 &\text{ if } \langle \tilde{z}, \tilde{x} \rangle<0 \text{ or } \left( \langle \tilde{z}, \tilde{x} \rangle=0 \text{ and }\tilde{z}_1=-B\tilde{x}_1\right)\\
\frac{1}{2^{d-1}} & \text{ if }\langle \tilde{z}, \tilde{x}\rangle>0 \text{ or } \left(\langle \tilde{z}, \tilde{x} \rangle=0 \text{ and }\tilde{z}_1=B\tilde{x}_1 \right),
\end{array}\right.
$$
and
$$
\Pb\left(\tilde{Z}=\tilde{z} \mid \tilde{X}=\tilde{x},Y=0 \right)=
\left\lbrace \begin{array}{ll}
\frac{1}{2^{d-1}} &\text{ if } \langle \tilde{z}, \tilde{x} \rangle<0 \text{ or } \left( \langle \tilde{z}, \tilde{x} \rangle=0 \text{ and }\tilde{z}_1=-B\tilde{x}_1\right)\\
0  & \text{ if } \langle \tilde{z}, \tilde{x}\rangle>0 \text{ or } \left(\langle \tilde{z}, \tilde{x} \rangle=0 \text{ and }\tilde{z}_1=B\tilde{x}_1 \right) .
\end{array}\right.
$$
We thus have
$$
\Pb\left(\tilde{Z}=\tilde{z} \mid \tilde{X}=\tilde{x} \right)=
\left\lbrace \begin{array}{ll}
\frac{1-\pi_\alpha}{2^{d-1}} &\text{ if } \langle \tilde{z}, \tilde{x} \rangle <0 \text{ or } \left( \langle \tilde{z}, \tilde{x} \rangle=0 \text{ and }\tilde{z}_1=-B\tilde{x}_1\right)\\
\frac{\pi_\alpha}{2^{d-1}}  & \text{ if } \langle \tilde{z}, \tilde{x} \rangle >0 \text{ or } \left(\langle \tilde{z}, \tilde{x} \rangle=0 \text{ and }\tilde{z}_1=B\tilde{x}_1 \right),
\end{array}\right.
$$
and, if we set
$$
A_{\tilde{z}}=\left\{\tilde{x}\in\{-1,1\}^d  : \langle \tilde{z}, \tilde{x} \rangle >0 \text{ or } \left(\langle \tilde{z}, \tilde{x} \rangle=0 \text{ and }\tilde{z}_1=B\tilde{x}_1 \right) \right\}
$$
and
$$
C_{\tilde{z}}=\left\{\tilde{x}\in\{-1,1\}^d  : \langle \tilde{z}, \tilde{x} \rangle <0 \text{ or } \left(\langle \tilde{z}, \tilde{x} \rangle=0 \text{ and }\tilde{z}_1=-B\tilde{x}_1 \right) \right\},
$$
 we obtain
$$
\Pb\left(\tilde{Z}=\tilde{z} \mid X=x \right)=\frac{\pi_\alpha}{2^{d-1}}\sum_{\tilde{x}\in A_{\tilde{z}}}\Pb\left( \tilde{X}=\tilde{x}\mid X=x \right)+\frac{1-\pi_\alpha}{2^{d-1}}\sum_{\tilde{x}\in C_{\tilde{z}}}\Pb\left( \tilde{X}=\tilde{x}\mid X=x \right).
$$
Consequently, it holds for all $\tilde{z}\in\{-B,B\}^d$ and all $x\in\Rb^d$,
$$
\frac{\min\{\pi_\alpha, 1-\pi_\alpha\}}{2^{d-1}}\leq \Pb\left(\tilde{Z}=\tilde{z} \mid X=x \right)\leq \frac{\max\{\pi_\alpha, 1-\pi_\alpha\}}{2^{d-1}},
$$
where we have used that $A_{\tilde{z}}\sqcup C_{\tilde{z}}=\{-1,1\}^d$ and $\sum_{\tilde{x}\in \{-1,1\}^d}\Pb\left( \tilde{X}=\tilde{x}\mid X=x \right)=1$.
We finally obtain for all $\tilde{z}\in\{-B,B\}^d$ and all $x,x'\in\Rb^d$,
$$
\frac{\Pb\left(\tilde{Z}=\tilde{z} \mid X=x \right)}{\Pb\left(\tilde{Z}=\tilde{z} \mid X=x' \right)}\leq \frac{\max\{\pi_\alpha, 1-\pi_\alpha\}}{\min\{\pi_\alpha, 1-\pi_\alpha\}}=\frac{\pi_\alpha}{1-\pi_\alpha}=e^\alpha.
$$

\subsection{Proof of Proposition \ref{Prop Expectation of the private data scenario 2}}\label{App Proof expectation of the private data scenario 2}

Let $x\in\Rb^d$.
We first compute $\Eb\left[\tilde{Z}\mid X=x  \right]$.
It holds
\begin{align*}
\Eb\left[\tilde{Z} \mid X=x \right]&=\sum_{\tilde{x}\in\{-1,1\}^d}\Pb\left( \tilde{X}=\tilde{x}\mid X=x \right)\cdot \Eb\left[\tilde{Z} \mid X=x, \tilde{X}=\tilde{x} \right] \\
&=\sum_{\tilde{x}\in\{-1,1\}^d}\Pb\left( \tilde{X}=\tilde{x}\mid X=x \right)\cdot \Eb\left[\tilde{Z}\mid \tilde{X}=\tilde{x} \right],
\end{align*}
and since $Y$ and $\tilde{X}$ are independent we have
\begin{align*}
\Eb\left[\tilde{Z} \mid \tilde{X}=\tilde{x} \right]&=\Pb\left(Y=1 \right)\cdot\Eb\left[\tilde{Z} \mid \tilde{X}=\tilde{x},Y=1 \right]+\Pb\left(Y=0 \right)\cdot \Eb\left[\tilde{Z} \mid \tilde{X}=\tilde{x}, Y=0 \right] \\
&=\pi_\alpha\Eb\left[\tilde{Z} \mid \tilde{X}=\tilde{x},Y=1 \right]+(1-\pi_\alpha)\Eb\left[\tilde{Z}=z \mid \tilde{X}=\tilde{x}, Y=0 \right].
\end{align*}
Define
\begin{align*}
&A_{\tilde{x}}:=\left\{\tilde{z}\in\left\{-B,B\right\}^d \, \mid \, \langle \tilde{z}, \tilde{x} \rangle >0 \text{ or } \left(\langle \tilde{z}, \tilde{x} \rangle=0 \text{ and }\tilde{z}_1=B\tilde{x}_1 \right)  \right\},\\
&C_{\tilde{x}}:=\left\{\tilde{z}\in\left\{-B,B\right\}^d \, \mid \, \langle \tilde{z}, \tilde{x} \rangle <0 \text{ or } \left(\langle \tilde{z}, \tilde{x} \rangle=0 \text{ and }\tilde{z}_1=-B\tilde{x}_1 \right)   \right\}.
\end{align*}
Conditionnally on $\left\{ \tilde{X}=\tilde{x},Y=1  \right\}$, it holds $Z\sim \Uc\left(A_{\tilde{x}} \right)$.
Thus,
\begin{align*}
\Eb\left[\tilde{Z} \mid \tilde{X}=\tilde{x},Y=1 \right]=\sum_{\tilde{z}\in A_{\tilde{x}}}\Pb\left( \tilde{Z}=\tilde{z} \mid \tilde{X}=\tilde{x}, Y=1 \right)\tilde{z}=\frac{1}{\Card(A_{\tilde{x}})}\sum_{\tilde{z}\in A_{\tilde{x}}}\tilde{z}.
\end{align*}
Similarly,
\begin{align*}
\Eb\left[\tilde{Z} \mid \tilde{X}=\tilde{x},Y=0 \right]&=\frac{1}{\Card(C_{\tilde{x}})}\sum_{\tilde{z}\in C_{\tilde{x}}}\tilde{z}=\frac{1}{\Card(C_{\tilde{x}})}\sum_{\tilde{z}\in A_{\tilde{x}}}(-\tilde{z})\\
&=-\Eb\left[\tilde{Z} \mid \tilde{X}=\tilde{x},Y=1 \right],
\end{align*}
where we have used $\Card(C_{\tilde{x}})=\Card(A_{\tilde{x}})$.
We thus obtain
$$
\Eb\left[\tilde{Z} \mid \tilde{X}=\tilde{x} \right]=\frac{2\pi_\alpha-1}{\Card(A_{\tilde{x}})}\sum_{\tilde{z}\in A_{\tilde{x}}}\tilde{z},
$$
and, using that $\Card(A_{\tilde{x}})=2^{d-1}$ for all $\tilde{x}\in\{-1,1\}^d$ we obtain
$$
\Eb\left[\tilde{Z} \mid X=x \right]=\frac{2\pi_\alpha-1}{2^{d-1}}\sum_{\tilde{x}\in\{-1,1\}^d}\left[\Pb\left( \tilde{X}=\tilde{x}\mid X=x \right)\cdot \sum_{\tilde{z}\in A_{\tilde{x}}}\tilde{z}\right].
$$
We now compute $\sum_{\tilde{z}\in A_{\tilde{x}}}\tilde{z}$ for all $\tilde{x}\in \{-1,1\}^d$.
Note that for $\tilde{z}\in\{-B,B\}^d$ and $\tilde{x}\in \{-1,1\}^d$, $\langle \tilde{z},\tilde{x}\rangle$ is a sum of $d$ terms, each equal to $-B$ or $B$.
If $a$ denotes the number of elements of this sum equal to $B$ and $b$ denotes the number of elements of this sum equal to $-B$, then it holds $a+b=d$ and $\langle \tilde{z},\tilde{x}\rangle=aB-bB=B(d-2b)$.
Thus we can only have $\langle \tilde{z},\tilde{x}\rangle=kB$, with $k\in \llbr -d,d\rrbr$ and $\vert k\vert$ has the same parity as $d$.
We thus have
\begin{equation}\label{Proof Expectation of the private data scenario 2, sum z when d odd}
\sum_{\tilde{z}\in A_{\tilde{x}}}\tilde{z}=\sum_{p=0}^{(d-1)/2}\sum_{\{\tilde{z}\in\{-B,B\}^d : \langle \tilde{z},\tilde{x}\rangle =(2p+1)B\}}z,
\end{equation}
if $d$ is odd, and
\begin{equation}\label{Proof Expectation of the private data scenario 2, sum z when d even}
\sum_{\tilde{z}\in A_{\tilde{x}}}\tilde{z}=\sum_{p=1}^{d/2}\displaystyle\sum_{\substack{ \tilde{z}\in\{-B,B\}^d: \\ \langle \tilde{z},\tilde{x}\rangle =2p\cdot B}}\tilde{z}+ \sum_{\substack{ \tilde{z}\in\{-B,B\}^d: \\ \langle \tilde{z},\tilde{x}\rangle=0 \\ \tilde{z}_1=B\tilde{x}_1}}\tilde{z},
\end{equation}
if $d$ is even.
Now, observe that for all $\tilde{x}\in\{-1,1\}^d$, for all $j\in\llbr 1,d\rrbr$ and for all $k\in\{0,\ldots,d\}$ with the same parity as $d$, it holds
\begin{equation}\label{Proof Expectation of the private data scenario 2, Intermediary result}
\sum_{\substack{ \tilde{z}\in\{-B,B\}^d: \\ \langle \tilde{z},\tilde{x}\rangle =kB}}\tilde{z}_j= B\left[ \binom{d-1}{\frac{d+k}{2}-1}-\binom{d-1}{\frac{d+k}{2}} \right]\tilde{x}_j.
\end{equation}
Indeed, for all $\tilde{z}\in\{-B,B\}^d$, for all $\tilde{x}\in\{-1,1\}^d$, and for all $k\in\{0,\ldots,d\}$ with the same parity as $d$, it holds
$$
\langle \tilde{z},\tilde{x}\rangle =k\cdot B \iff
\left\lbrace \begin{array}{ll}
\tilde{z}_j=B\tilde{x}_j &\text{ for $\frac{d+k}{2}$ elements }  j\in\llbr 1,d\rrbr\\
\tilde{z}_j=-B\tilde{x}_j  & \text{ for $\frac{d-k}{2}$ elements } j\in\llbr 1,d\rrbr.
\end{array}\right.
$$
Setting $D_{k,\tilde{x}}=\{\tilde{z}\in\{-B,B\}^d : \langle \tilde{z},\tilde{x}\rangle =k\cdot B\}$, it thus holds
\begin{align*}
\sum_{\tilde{z}\in D_{k,\tilde{x}}}\tilde{z}_j&=\sum_{\tilde{z}\in D_{k,\tilde{x}}}B\tilde{x}_j\1\left(\tilde{z}_j=B\tilde{x}_j \right)-\sum_{z\in D_{k,\tilde{x}} }B\tilde{x}_j\1\left(\tilde{z}_j=-B\tilde{x}_j \right)\\
&=B\left[\Card\left(\tilde{z}\in D_{k,\tilde{x}} : \tilde{z}_j=B\tilde{x}_j \right) -\Card\left(\tilde{z}\in D_{k,\tilde{x}} : \tilde{z}_j=-B\tilde{x}_j \right) \right] \tilde{x}_j\\
&=B\left[ \binom{d-1}{\frac{d+k}{2}-1}-\binom{d-1}{\frac{d+k}{2}} \right]\tilde{x}_j.
\end{align*}
We now end the proof of Proposition \ref{Prop Expectation of the private data scenario 2} when $d$ is odd.
Combining \eqref{Proof Expectation of the private data scenario 2, Intermediary result} with \eqref{Proof Expectation of the private data scenario 2, sum z when d odd}, we obtain for $d$ odd
 $$
\sum_{z\in A_{\tilde{x}}}\tilde{z}=B\binom{d-1}{\frac{d-1}{2}}\tilde{x},
$$
and the choice of $B$ yields
\begin{align*}
\Eb\left[\tilde{Z} \mid X=x \right]&=\frac{2\pi_\alpha-1}{2^{d-1}}B\binom{d-1}{\frac{d-1}{2}}\sum_{\tilde{x}\in\{-1,1\}^d}\Pb\left(\tilde{X}=\tilde{x} \mid X=x \right)\cdot \tilde{x}\\
&=\Eb\left[\tilde{X}\mid X=x  \right]
\end{align*}
Since for all $j\in\llbr 1,d\rrbr$ it holds
$$
\Eb\left[ \tilde{X}_j \mid X=x \right]=\sgn[x_j],
$$
we obtain for $d$ odd
$$
\Eb\left[Z \mid X=x \right]=\Eb\left[\tilde{Z} \mid X=x \right]=\Eb\left[\tilde{X}\mid X=x  \right]=f(x),
$$
which proves Proposition \ref{Prop Expectation of the private data scenario 2} when $d$ is odd.
From now on, we assume that $d$ is even.
Combining \eqref{Proof Expectation of the private data scenario 2, Intermediary result} with  \eqref{Proof Expectation of the private data scenario 2, sum z when d even}, we obtain
$$
\sum_{\tilde{z}\in A_{\tilde{x}}}\tilde{z}=B\binom{d-1}{\frac{d}{2}}\tilde{x} + \sum_{\substack{ \tilde{z}\in\{-B,B\}^d: \\ \langle \tilde{z},\tilde{x}\rangle=0 \\ \tilde{z}_1=B\tilde{x}_1}}\tilde{z}.
$$
Now, observe that for $\tilde{z}\in\{-B,B\}^d$ and $\tilde{x}\in\{-1,1\}^d$ it holds $\langle \tilde{z},\tilde{x}\rangle=0$ if and only if $\tilde{z}_j=B\tilde{x}_j$ for exactly $d/2$ subscripts $j\in\llbr 1,d\rrbr$ and $\tilde{z}_j=-B\tilde{x}_j$ for exactly $d/2$ subscripts $j\in\llbr 1,d\rrbr$.
We thus have
\begin{align*}
\sum_{\substack{ \tilde{z}\in\{-B,B\}^d: \\ \langle \tilde{z},\tilde{x}\rangle=0 \\ \tilde{z}_1 =B\tilde{x}_1}}\tilde{z}_1 &=B\tilde{x}_1\cdot\Card\left(\left\{ \tilde{z}\in\{-B,B\}^d : \langle \tilde{z},\tilde{x}\rangle=0 \text{ and }\tilde{z}_1=B\tilde{x}_1 \right\} \right)\\
&=B\binom{d-1}{\frac{d}{2}-1}\tilde{x}_1,
\end{align*}
and for $j\geq 2$ it holds
\begin{align*}
\sum_{\substack{ \tilde{z}\in\{-B,B\}^d: \\ \langle \tilde{z},\tilde{x}\rangle=0 \\ \tilde{z}_1 =B\tilde{x}_1}}\tilde{z}_j &= B\tilde{x}_j\left[\Card\left(\left\{ \tilde{z}\in\{-B,B\}^d : \langle \tilde{z},\tilde{x}\rangle=0, \tilde{z}_1=B\tilde{x}_1, \tilde{z}_j=B\tilde{x}_j \right\} \right)\right.\\
&\hspace{2cm}\left. -\Card\left(\left\{ \tilde{z}\in\{-B,B\}^d : \langle \tilde{z},\tilde{x}\rangle=0, \tilde{z}_1=B\tilde{x}_1, \tilde{z}_j=-B\tilde{x}_j \right\} \right) \right]\\
&=B\left[\binom{d-2}{\frac{d}{2}-2}-\binom{d-2}{\frac{d}{2}-1}  \right]\tilde{x}_j.
\end{align*}
We thus obtain
$$
\sum_{\tilde{z}\in A_{\tilde{x}}}\tilde{z}_j=
\begin{cases}
B\binom{d}{\frac{d}{2}}\tilde{x}_1 & \text{ if } j=1\\
B\left[\binom{d-1}{\frac{d}{2}}+\binom{d-2}{\frac{d}{2}-2} -\binom{d-2}{\frac{d}{2}-1}  \right]\tilde{x}_j & \text{ if } j\in\llbr 2,d\rrbr
\end{cases}.
$$
The choice
$$
B=\frac{2^{d-1}}{2\pi_\alpha-1}\cdot \frac{(\frac{d}{2}-1)!\frac{d}{2}!}{(d-2)!(d-2)}
$$
then yields
\begin{align*}
\Eb\left[\tilde{Z}_j \mid X=x \right]&=
\begin{cases}
\frac{(2\pi_\alpha-1)B}{2^{d-1}}\binom{d}{\frac{d}{2}}\sum_{\tilde{x}\in\{-1,1\}^d}\tilde{x}_1\Pb\left( \tilde{X}=\tilde{x}\mid X=x \right) & \text{ if } j=1\\
\frac{(2\pi_\alpha-1)B}{2^{d-1}}\cdot \frac{(d-2)!(d-2)}{(\frac{d}{2}-1)!\frac{d}{2}!} \sum_{\tilde{x}\in\{-1,1\}^d}\tilde{x}_j\Pb\left( \tilde{X}=\tilde{x}\mid X=x \right) & \text{ if } j\in\llbr 2,d\rrbr
\end{cases}\\
&=\begin{cases}
\frac{2(d-1)}{d-2}\sum_{\tilde{x}\in\{-1,1\}^d}\tilde{x}_1\Pb\left( \tilde{X}=\tilde{x}\mid X=x \right) & \text{ if } j=1\\
\sum_{\tilde{x}\in\{-1,1\}^d}\tilde{x}_j\Pb\left( \tilde{X}=\tilde{x}\mid X=x \right) & \text{ if } j\in\llbr 2,d\rrbr
\end{cases}.
\end{align*}
Thus, it holds $\Eb\left[Z_j \mid X=x \right]=\sum_{\tilde{x}\in\{-1,1\}^d}\Pb\left( \tilde{X}=\tilde{x}\mid X=x \right)\tilde{x}_j$ for all $j\in\llbr 1,d \rrbr$, and
$$
\Eb[Z\mid X=x]=\sum_{\tilde{x}\in\{-1,1\}^d}\Pb\left( \tilde{X}=\tilde{x}\mid X=x \right)\tilde{x}=\Eb\left[\tilde{X}\mid X=x\right]=f(x).
$$

\subsection{Asymptotic analysis of the value $K_d$ defined in \eqref{Kd}}\label{App Proof Asymptotic Property Kd}

\begin{lemma}\label{Lem Asymptotic Property Kd}
The value $K_d$ defined in (\ref{Kd}) behaves asymptotically in $d$ as
$$
K_d\underset{d\rightarrow \infty}{\sim} \sqrt{\frac{\pi}{2}}\sqrt{d}.
$$
In particular, it holds $K_d\lesssim  \sqrt{d}$ for $d$ large enough.
\end{lemma}

The proof relies on Stirling's approximation.
We first deal with the case where $d$ is odd.
In this case, Stirling's approximation yields
$$
K_d=2^{d-1}\frac{\left[\left(\frac{d-1}{2} \right)!\right]^2}{(d-1)!}\underset{d\rightarrow \infty}{\sim} 2^{d-1}\cdot \pi(d-1)\left(\frac{d-1}{2e} \right)^{d-1}\cdot \left[\sqrt{2\pi (d-1)}\left(\frac{d-1}{e} \right)^{d-1}  \right]^{-1}.
$$
The right-hand side of this asymptotic equivalence is equal to $\sqrt{\pi/2}\sqrt{d-1}$.
We thus obtain $K_d\underset{d\rightarrow \infty}{\sim}\sqrt{\pi/2}\sqrt{d}$.

We now assume that $d$ is even.
in this case, Stirling's approximation yields
$$
K_d= \frac{2^{d-1}(\frac{d}{2}-1)!\frac{d}{2}!}{(d-2)!(d-2)}\underset{d\rightarrow \infty}{\sim}\frac{2^{d-1}}{d-2}\cdot \pi\sqrt{(d-2)d}\left(\frac{d-2}{2e} \right)^{\frac{d}{2}-1}\left( \frac{d}{2e}\right)^{\frac{d}{2}}\cdot \left[ \sqrt{2\pi(d-2)}\left(\frac{d-2}{e} \right)^{d-2}\right]^{-1}
$$
The right-hand side of this asymptotic equivalence is equal to
$$
\frac{\sqrt{\pi}}{e\sqrt{2}}\sqrt{d}(d-2)^{-\frac{d}{2}}d^{\frac{d}{2}}=\frac{\sqrt{\pi}}{e\sqrt{2}}\sqrt{d}\exp\left(-\frac{d}{2}\log\left(1-\frac{2}{d} \right) \right)\underset{d\rightarrow \infty}{\sim}\sqrt{\frac{\pi}{2}}\sqrt{d},
$$
which ends the proof.

\subsection{Proof of Proposition \ref{Prop Upper Bound Large a scenario 2}}\label{App Proof upper bound large a scenario 2}

The proof is similar to the one we made in the Coordinate Local case (Proposition \ref{Prop Upper Bound on the risk for large a Scenario 1}).
However, in the Coordinate Global case, for all $j\in\llbr 1,d\rrbr$ the $(Z^{i}_j)_i$ are bounded random variables, which will enable us to use Hoeffding's inequality instead of Lemma \ref{Lemme Troncation Sous Gaussienne} and Bernstein's inequality.

Writting
$$
\vert \hat{\eta}^+-\eta\vert=\sum_{j : \eta_j=0}\hat{\eta}^+_j+\sum_{j : \eta_j=1}(1-\hat{\eta}^+_j),
$$
we have
$$
\Eb\left[\frac{1}{s}\vert\hat{\eta}^+-\eta\vert \right]=\frac{1}{s}\sum_{j : \eta_j=0}\underbrace{\Pb\left(\frac{1}{n}\sum_{i=1}^nZ^{i}_j\geq \tau\right)}_{=T_{1,j}}+\frac{1}{s}\sum_{j : \eta_j=1}\underbrace{\Pb\left(\frac{1}{n}\sum_{i=1}^nZ^{i}_j < \tau\right)}_{=T_{2,j}}.
$$
We first study $T_{1,j}$.
For $j$ satisfying $\eta_j=0$, it holds
$$
\Eb\left[ Z^{i}_j \right]=\Eb\left[\Eb\left[ Z^{i}_j \mid X^{i}\right]\right]=\Eb\left[ \sgn[X^{i}_j] \right]=\Eb\left[ \sgn[\sigma\xi^{i}_j] \right]=0,
$$
where we have used Proposition \ref{Prop Expectation of the private data scenario 2} and the fact that the distribution of the random variable $\xi_j^i$ is symmetric.
Thus, Hoeffding's inequality yields
$$
T_{1,j}=\Pb\left(\sum_{i=1}^n(Z^{i}_j-\Eb[ Z^{i}_j])\geq n\tau\right)\leq \exp\left(-\frac{n\tau^2}{2B^2} \right).
$$
We now study $T_{2,j}$.
Let $j\in\llbr 1,d\rrbr$ such that $\eta_j=1$.
It holds
\begin{align*}
T_{2,j}&=\Pb\left(\frac{1}{n}\sum_{i=1}^n\left(Z^{i}_j- \Eb\left[Z^{i}_j \right] \right)+\frac{1}{n}\sum_{i=1}^n\Eb\left[Z^{i}_j \right]< \tau\right)\\
&=\Pb\left(\frac{1}{n}\sum_{i=1}^n\left(-Z^{i}_j- \Eb\left[-Z^{i}_j \right] \right)>\Eb\left[Z^{1}_j \right] - \tau\right).
\end{align*}
Proposition \ref{Prop Expectation of the private data scenario 2} gives
$$
\Eb\left[Z^{1}_j \right] =\Eb\left[\sgn[X^{1}_j ]\right] =\Eb\left[\sgn[\theta_j+\sigma \xi_j^1 ]\right]\geq \Eb\left[\sgn[a+\sigma \xi_j^1 ]\right],
$$
and  we have proved in Appendix \ref{App proof Upper Bound Large a scenario 1} that it holds
$$
\Eb\left[\sgn[a+\sigma \xi_1]\right]\geq 2\Phi\left(\sqrt{c}\frac{a}{\sigma}\right)-1,
$$
where $\Phi$ denotes the standard Gaussian cumulative distribution function.
Thus, if $a\geq 2\sigma$, it holds $\Eb\left[\sgn[a+\sigma \xi^{1}_j]\right]\geq C_1$ with $C_1=2\Phi(2\sqrt{c})-1$, and
\begin{align*}
T_{2,j}&\leq \Pb\left(\frac{1}{n}\sum_{i=1}^n\left(-Z^{i}_j- \Eb\left[-Z^{i}_j \right] \right)>C_1 - \tau\right)\\
&\leq \exp\left(-\frac{n(C_1-\tau)^2}{2B^2} \right)
\end{align*}
according to Hoeffding's inequality if $C_1-\tau>0$.
This yields
$$
\Eb\left[\frac{1}{s}\vert\hat{\eta}^+-\eta\vert \right]\leq \frac{d-\vert S\vert}{s}\exp\left(-\frac{n\tau^2}{2B^2 }\right)+\frac{\vert S\vert}{s}\exp\left(-\frac{n(C_1-\tau)^2}{2B^2}\right).
$$
The proof of the second statement of Proposition \ref{Prop Upper Bound Large a scenario 2} is straightforward.

\subsection{Proof of Proposition \ref{Prop Upper Bound Small a scenario 2}}\label{App Proof upper bound small a scenario 2}

The beginning of the proof is similar to the proof of Proposition \ref{Prop Upper Bound Large a scenario 2}.
It holds
$$
\Eb\left[\frac{1}{s}\vert\hat{\eta}^+-\eta\vert \right]=\frac{1}{s}\sum_{j : \eta_j=0}\underbrace{\Pb\left(\frac{1}{n}\sum_{i=1}^nZ^{i}_j\geq \tau\right)}_{=T_{1,j}}+\frac{1}{s}\sum_{j : \eta_j=1}\underbrace{\Pb\left(\frac{1}{n}\sum_{i=1}^nZ^{i}_j < \tau\right)}_{=T_{2,j}},
$$
with
$$
T_{1,j}\leq \exp\left(-\frac{n\tau^2}{2B^2} \right),
$$
and
$$
T_{2,j}\leq \Pb\left(\frac{1}{n}\sum_{i=1}^n\left(-Z^{i}_j- \Eb\left[-Z^{i}_j \right] \right)>\Eb\left[\sgn[a+\sigma \xi_j^{i}] \right] - \tau\right).
$$
Moreover, we have proved in Appendix \ref{App Proof Upper Bound Small a scenario 1} that  $\Eb\left[\sgn[a+\sigma \xi_1]\right]\geq 2p(2)a/\sigma$ for $a/\sigma<2$.
Thus, if $\tau< 2p(2)a/\sigma$, Hoeffding's inequality yields
$$
T_{2,j}\leq \Pb\left(\frac{1}{n}\sum_{i=1}^n\left(-Z^{i}_j- \Eb\left[-Z^{i}_j \right] \right)>\frac{2p(2)a}{\sigma} - \tau\right)\leq \exp\left(-\frac{n(2p(2)a/\sigma-\tau)^2}{2B^2} \right).
$$

\subsection{Proof of Proposition \ref{Prop Lower Bound scenario 2}}\label{App Proof Lower bound scenario 2}

For $i=1,\ldots, d$, define the vector $\omega_i\in\{0,1\}^d$ by $\omega_{i,j}=1$ if $j=i$, $\omega_{i,j}=0$ if $j\neq i$ and define $P_{\omega_i}$ as the multivariate distribution of the random vector $X=a\omega_i+\sigma\xi$.
For $i\neq j$ it holds
$$
\vert \eta(P_{\omega_i})- \eta(P_{\omega_j})\vert = \vert \omega_i-\omega_j\vert =2.
$$
The private Fano method (Proposition 2 in \cite{DuchiJordanWainwright2018MinimaxOptimalProcedure}) thus yields
$$
\inf_{Q\in \Qc_{\alpha}}\inf_{\hat{\eta}\in\Tc}\sup_{\theta\in\Theta_d^+(s,a)}\Eb_{Q(P_\theta^{\otimes n})}\vert\hat{\eta}-\eta\vert\geq \frac{1}{2}\left\{ 1-\frac{n(e^\alpha-1)^2}{d\log(d)}\left[ \sup_{\gamma\in \Bb_\infty(\Rb^d)}\sum_{i=1}^d(\varphi_{\omega_i}(\gamma))^2 \right] -\frac{\log(2)}{\log(d)}\right\},
$$
with
$$
\Bb_\infty(\Rb^d)=\left\{\gamma\in L_\infty(\Rb^d) \mid \Vert\gamma\Vert_\infty\leq 1 \right\},
$$
$$
\varphi_{\omega_i}(\gamma)=\int_{\Xc}\gamma(x)(d P_{\omega_i}(x)-d\bar{P}(x))=\int_{\Rb^d}\gamma(x)(f_{\omega_i}(x)-\bar{f}(x))dx,
$$
where $f_{\omega_i}$ is the density of $P_{\omega_i}$ and $\bar{f}=(1/d)\sum_{i=1}^df_{\omega_i}$.
We have
\begin{align*}
\sum_{i=1}^d \left( \varphi_{\omega_i}(\gamma)\right)^2&=\sum_{i=1}^d\left(\int_{\Rb^d}\gamma(x)(f_{\omega_i}(x)-\bar{f}(x))dx \right)\left(\int_{\Rb^d}\gamma(y)(f_{\omega_i}(y)-\bar{f}(y))dy \right)\\
&= \int_{\Rb^d}\gamma(x)\left[\int_{\Rb^d}\left(\sum_{i=1}^d(f_{\omega_i}(x)-\bar{f}(x))(f_{\omega_i}(y)-\bar{f}(y)) \right)\gamma(y) dy \right]dx.
\end{align*}
Let $\bar{p}$ denote the density of the random vector $\sigma \xi$.
If $\gamma$ belongs to $ \Bb_\infty(\Rb^d)$ then it also belongs to $ L_2(\Rb^d, dq)$ and, moreover, $\Vert \gamma\Vert_{L_2(\Rb^d, d\bar{p})}\leq 1$. We can write
\begin{align*}
\sum_{i=1}^d \left( \varphi_{\omega_i}(\gamma)\right)^2&=\int_{\Rb^d}\gamma(x)\left[\int_{\Rb^d}\left(\sum_{i=1}^d\frac{f_{\omega_i}(x)-\bar{f}(x)}{\bar{p}(x)}\cdot\frac{f_{\omega_i}(y)-\bar{f}(y)}{\bar{p}(y)} \right)\gamma(y)\bar{p}(y) dy \right]\bar{p}(x)dx\\
&= \langle \gamma, K\gamma \rangle_{L_2(\Rb^d, d\bar{p})},
\end{align*}
where
$$
\begin{array}{ccccc}
K & : & L_2(\Rb^d, d\bar{p}) & \to & L_2(\Rb^d, d\bar{p})\\
 & & \gamma & \mapsto & \int_{\Rb^d}\left(\sum_{i=1}^d\frac{f_{\omega_i}-\bar{f}}{\bar{p}}(\cdot)\cdot\frac{f_{\omega_i}(y)-\bar{f}(y)}{\bar{p}(y)} \right)\gamma(y)\bar{p}(y) dy \\
\end{array}
$$
For any $\omega\in \{0,1\}^d$, $f_{\omega}\in L_2(\Rb^d, d\bar{p})$. Note that we can rewrite
$$
K\gamma= \sum_{i=1}^d\left[\left\langle \frac{f_{\omega_i}-\bar{f}}{\bar{p}},\gamma \right\rangle_{L_2(\Rb^d, d\bar{p})}\cdot \frac{f_{\omega_i}-\bar{f}}{\bar{p}}\right].
$$
This expression implies that  $K$ is an operator of finite rank (it is thus a compact operator), $K$ is self-adjoint, and $\langle K\gamma, \gamma\rangle \geq 0$ for all $\gamma \in L_2(\Rb^d, d\bar{p})$. In particular, the last point implies that the eigenvalues of $K$ are non-negative.
We have
\begin{align*}
\sup_{\gamma\in\Bb_\infty(\Rb^d)}\sum_{i=1}^d \left( \varphi_{\omega_i}(\gamma)\right)^2 &\leq  \sup_{\{\gamma\in L_2(\Rb^d, d\bar{p}): \Vert \gamma\Vert_{L_2(\Rb^d, d\bar{p})}^2\leq 1\}}\langle \gamma, K\gamma \rangle_{L_2(\Rb^d, d\bar{p})}\\
&=\sup_{\{\gamma\in L_2(\Rb^d, d\bar{p}): \Vert \gamma\Vert_{L_2(\Rb^d, d\bar{p})}^2=1\}}\langle \gamma, K\gamma \rangle_{L_2(\Rb^d, d\bar{p})}\\
&=\sup_{\{\gamma\in L_2(\Rb^d, d\bar{p}): \Vert \gamma\Vert_{L_2(\Rb^d, d\bar{p})}^2=1\}}\left\vert\langle  \gamma, K\gamma \rangle_{L_2(\Rb^d, d\bar{p})}\right\vert\\
&=\Vert K\Vert,
\end{align*}
where the last equality follows from the fact that $\left(L_2(\Rb^d, d\bar{p}), \langle \cdot,\cdot\rangle_{L_2(\Rb^d, d\bar{p})}\right)$ is an Hilbert space and $K$ is self-adjoint.
Since $K$ is also compact and since the eigenvalues of $K$ are non-negative it follows
$$
\sup_{\gamma\in\Bb_\infty(\Rb^d)}\sum_{i=1}^d \left( \varphi_{\omega_i}(\gamma)\right)^2\leq \Vert K\Vert =\max \{ \vert \lambda \vert : \lambda \in VP(T)\}=\max \{  \lambda  : \lambda \in VP(T)\},
$$
where $VP(T)$ is the set of all the eigenvalues of $K$.
It remains to compute this maximum.
By definition, $\lambda$ is an eigenvalue of $K$ if $\lambda I-K$ is not injective.
For $\lambda\neq 0$, the Fredholm alternative for compact self-adjoint operators (see for instance  \cite{Hirsch_Lacombe_analyse_fonctionnelle}) implies that $\lambda I-K$ is not injective if and only if $\lambda I-K$ is not surjective.
Thus, the non-zero eigenvalues of $K$ are the values of $\lambda \in\Rb^*$ such that the operator $\lambda I-K$ is not surjective.
For $\lambda\in\Rb$,  let $A_{\lambda}$ be the matrix with coefficients
$$
(A_\lambda)_{ij}=\left\langle \frac{f_{\omega_i}-\bar{f}}{\bar{p}},\frac{f_{\omega_j}-\bar{f}}{\bar{p}}\right\rangle_{L_2(\Rb^d, d\bar{p})}-\lambda \delta_{ij}, \quad i,j\in\llbr 1,d\rrbr,
$$
where $\delta$ is the Kronecker delta.
The following result proves that if $\lambda$ is a non-zero eigenvalue of $K$ then it holds $\Det(A_\lambda)= 0$.
\begin{lemma}
Let $\lambda\in\Rb$, $\lambda\neq 0$. If $\Det(A_\lambda)\neq 0$ then $\lambda I-K$ is surjective.
\end{lemma}
\begin{proof}
To lighten the notation, set $\langle \cdot,\cdot\rangle_{2,\bar{p}}=\langle \cdot,\cdot\rangle_{L_2(\Rb^d, d\bar{p})}$.
Let $\lambda\in\Rb$, $\lambda\neq 0$ and assume that $\Det(A_\lambda)\neq 0$.
We prove that for all $g\in L_2(\Rb^d,d\bar{p})$ there exists $\gamma\in L_2(\Rb^d,d\bar{p})$ such that $g=(\lambda I-K)\gamma$.
Consider $g\in L_2(\Rb^d,d\bar{p})$.
Since $\Det(A_\lambda)\neq 0$, the matrix $A_\lambda$ is invertible and for all $v\in \Rb^d$ there exists $\xi\in\Rb^d$ such that $v=A_\lambda \xi$. In particular, for
$$
v=\left( \left\langle \frac{f_{\omega_1}-\bar{f}}{\bar{p}},g\right\rangle_{2,\bar{p}},\ldots, \left\langle\frac{f_{\omega_d}-\bar{f}}{\bar{p}},g\right\rangle_{2,\bar{p}}\right)^T,
$$
there exists $\xi\in\Rb^d$ such that $v=A_\lambda \xi$, that is
$$
\left\langle \frac{f_{\omega_i}-\bar{f}}{\bar{p}},g\right\rangle_{2,\bar{p}}=(A_\lambda\xi)_i=\sum_{j=1}^d \left\langle \frac{f_{\omega_i}-\bar{f}}{\bar{p}},\frac{f_{\omega_j}-\bar{f}}{\bar{p}}\right\rangle_{2,\bar{p}}\xi_j-\lambda \xi_i
$$
for all $i\in\llbr 1,d\rrbr$.
Define
$$
\gamma=\frac{1}{\lambda}g-\frac{1}{\lambda}\sum_{j=1}^d\xi_j\frac{f_{\omega_j}-\bar{f}}{\bar{p}}.
$$
We have
\begin{align*}
(\lambda I-K)\gamma&=\lambda \gamma-K\gamma\\
&=g-\sum_{i=1}^d\xi_i\frac{f_{\omega_i}-\bar{f}}{\bar{p}}-\sum_{i=1}^d\left[\left\langle \frac{f_{\omega_i}-\bar{f}}{\bar{p}},\gamma \right\rangle_{L_2(\Rb^d, d\bar{p})}\cdot \frac{f_{\omega_i}-\bar{f}}{\bar{p}}\right]\\
&=g-\sum_{i=1}^d\underbrace{\left[\xi_i+  \frac{1}{\lambda}\left\langle \frac{f_{\omega_i}-\bar{f}}{\bar{p}},g\right\rangle_{2,\bar{p}}-\frac{1}{\lambda}\sum_{j=1}^d \xi_j\left\langle \frac{f_{\omega_i}-\bar{f}}{\bar{p}},\frac{f_{\omega_j}-\bar{f}}{\bar{p}}\right\rangle_{2,\bar{p}} \right]}_{=0}\frac{f_{\omega_i}-\bar{f}}{\bar{p}}\\
&=g,
\end{align*}
which concludes the proof of the Lemma.
\end{proof}
We now find the values of $\lambda$ for which we have $\Det(A_\lambda)= 0$.
To do so, we first make explicit the coefficients of $A_\lambda$.
It holds
\begin{align*}
\left\langle \frac{f_{\omega_i}-\bar{f}}{\bar{p}},\frac{f_{\omega_j}-\bar{f}}{\bar{p}}\right\rangle_{2,\bar{p}}&=\left\langle \frac{f_{\omega_i}}{\bar{p}},\frac{f_{\omega_j}}{\bar{p}}\right\rangle_{2,\bar{p}}-\left\langle \frac{f_{\omega_i}}{\bar{p}},\frac{\bar{f}}{\bar{p}}\right\rangle_{2,\bar{p}}-\left\langle \frac{\bar{f}}{\bar{p}},\frac{f_{\omega_j}}{\bar{p}}\right\rangle_{2,\bar{p}}+\left\langle \frac{\bar{f}}{\bar{p}},\frac{\bar{f}}{\bar{p}}\right\rangle_{2,\bar{p}}\\
&=\left\langle \frac{f_{\omega_i}}{\bar{p}},\frac{f_{\omega_j}}{\bar{p}}\right\rangle_{2,\bar{p}}-\frac{1}{d}\sum_{k=1}^d \left\langle \frac{f_{\omega_i}}{\bar{p}},\frac{f_{\omega_k}}{\bar{p}}\right\rangle_{2,\bar{p}}-\frac{1}{d}\sum_{k=1}^d\left\langle \frac{f_{\omega_k}}{\bar{p}},\frac{f_{\omega_j}}{\bar{p}}\right\rangle_{2,\bar{p}}\\
&\hspace{2,5cm}+\frac{1}{d^2}\sum_{k=1}^d\sum_{l=1}^d\left\langle \frac{f_{\omega_k}}{\bar{p}},\frac{f_{\omega_l}}{\bar{p}}\right\rangle_{2,\bar{p}}.
\end{align*}
{Furthermore, due to the independence of the coordinates of the vector $\xi$, the scalar products $\left\langle \frac{f_{\omega_k}}{\bar{p}},\frac{f_{\omega_l}}{\bar{p}}\right\rangle_{2,\bar{p}}$ can only take two values. More precisely, recall that $P_0$ denotes the distribution of the random variable $\sigma\xi_1$ and $P_a$ the distribution of the random variable $a+\sigma\xi_1$, we get
\begin{align*}
\left\langle \frac{f_{\omega_i}}{\bar{p}},\frac{f_{\omega_j}}{\bar{p}}\right\rangle_{2,\bar{p}}
&=
\begin{cases}1 + \chi^2 (P_0,P_a)
& \text{if } j=i\\
1 & \text{if } j\neq i.
\end{cases}
\end{align*}

We thus obtain
$$
\left\langle \frac{f_{\omega_i}-\bar{f}}{\bar{p}},\frac{f_{\omega_j}-\bar{f}}{\bar{p}}\right\rangle_{2,\bar{p}}=
\begin{cases}
\left(1-\frac{1}{d} \right)\chi^2 (P_0,P_a) & \text{if } j=i\\
-\frac{1}{d} \chi^2(P_0,P_a) & \text{if } j\neq i.
\end{cases}
$$
Write
$$
C_1= \left(1-\frac{1}{d} \right)\chi^2(P_0,P_a)  ,
$$
and
$$
C_2=- \frac{1}{d} \chi^2(P_0,P_a).
$$
The matrix $A_\lambda$ has its diagonal elements equal to $C_1-\lambda$ and the other coefficients equal to $C_2$.
Operations on the rows and columns of $A_\lambda$ yield
\begin{align*}
\Det(A_\lambda)&= \left(C_1+(d-1)C_2-\lambda  \right)\left(C_1-C_2-\lambda \right)^{d-1}\\
&=-\lambda\left(\chi^2(P_0,P_a) -\lambda \right)^{d-1}
\end{align*}
Thus, the operator $K$ has only one non-zero eigenvalue and it is equal to  $\chi^2(P_0,P_a)$.
We finally obtain
\begin{align*}
\inf_{Q\in \Qc_{\alpha}}\inf_{\hat{\eta}\in\Tc}\sup_{\theta\in\Theta_d^+(s,a)}\Eb_{Q(P_\theta^{\otimes n})}\vert\hat{\eta}-\eta\vert&\geq \frac{1}{2}\left(1-\frac{n(e^\alpha-1)^2}{d\log(d)}
\chi^2(P_0,P_a) -\frac{\log(2)}{\log(d)}\right)\\
&\geq \frac{1}{4}\left(1- \frac{2n(e^\alpha-1)^2}{d\log(d)} \chi^2(P_0,P_a) \right),
\end{align*}
if $d \geq 4$.
To conclude with the proof of Proposition \ref{Prop Lower Bound scenario 2}, just use Lemma \ref{lem:bound_chi2} below.
}

\begin{lemma}
\label{lem:bound_chi2}Consider that the measure $P^{\xi_1}$ of the noise coordinates has a density $p=\exp(-\phi)$, where the potential
$\phi$ is two times continuously differentiable and has a curvature that is bounded from above by a constant $c_+$ as in \eqref{eq:upper_bound_hessian_intro}. Then it holds
\begin{equation}\label{eq:lem_chi2}
    \chi^{2}\left(P_{0},P_{a}\right)\leq\exp\left(c_+\left(\frac{a}{\sigma}\right)^{2}\right)-1.
\end{equation}
If the density $p$ is log-concave, with a potential with curvature bounded above by $c_+$ as in (3), then Inequality (\ref{eq:lem_chi2}) holds without assuming the differentiability of $\phi$. 
\end{lemma}
Note that Lemma \ref{lem:bound_chi2} is sharp in the sense that in the Gaussian case, $c_+=1$ holds and Inequality \eqref{eq:lem_chi2} turns out to be an equality. Note also that log-concavity is actually not needed in Lemma \ref{lem:bound_chi2}, since we only require an upper bound on the curvature of the potential $\phi$.
\begin{proof}
Denote $\bar{a}=a/\sigma$. It suffices to show the following inequality,
\begin{equation}
\int_{\mathbb{R}}\frac{p^{2}(x-\bar{a})}{p(x)}dx\leq\exp\left(c_{+}\bar{a}^{2}\right),\label{eq:goal_chi2}
\end{equation}
where we recall that $p$ is the density of $\xi_1$. It holds
\begin{align*}
\int_{\mathbb{R}}\frac{p^{2}(x-\bar{a})}{p(x)}dx & =\int_{\mathbb{R}}\exp\left(\phi(x)-2\phi(x-\bar{a})\right)dx.
\end{align*}
 As $\phi$ is two times continuously differentiable, we have by Taylor
expansion, for all $x\in\mathbb{R}$,
\[
\phi(x)-\phi(x-\bar{a})\leq\bar{a}\phi^{\prime}(x-\bar{a})+c_{+}\frac{\bar{a}^{2}}{2}
\]
and
\[
\phi(x-2\bar{a})-\phi(x-\bar{a})\leq-\bar{a}\phi^{\prime}(x-\bar{a})+c_{+}\frac{\bar{a}^{2}}{2}.
\]
By adding the two previous inequalities, we get
\[
\phi(x)-2\phi(x-\bar{a})\leq-\phi(x-2\bar{a})+c_{+}\bar{a}^{2}.
\]
This gives
\begin{align*}
\int_{\mathbb{R}}\frac{p^{2}(x-\bar{a})}{p(x)}dx & =\int_{\mathbb{R}}\exp\left(\phi(x)-2\phi(x-\bar{a})\right)dx\\
 & \leq\exp\left(c_{+}\bar{a}^{2}\right)\int_{\mathbb{R}}\left(\exp\left(-\phi(x-2\bar{a})\right)\right)dx\\
 & =\exp\left(c_{+}\bar{a}^{2}\right).
\end{align*}
We proved \eqref{eq:goal_chi2}. In the case where $p$ is log-concave, it can be suitably approximated by infinitely differentiable densities, via the use of convultions with Gaussian random variables,
which completes the proof of Lemma \ref{lem:bound_chi2}.
\end{proof}

\printbibliography

\end{document}